\documentclass[10pt]{article}
\usepackage{amsmath, amssymb ,amsthm, amsfonts, amsgen}
\usepackage{graphicx}
\usepackage{rotate}
\usepackage[dvips]{color}

\numberwithin{equation}{section} 

\newcommand{\Rb}{{\mathbb{R}}}

\newcommand{\C}{{\mathcal{C}}}

\newcommand{\LL}{{\mathcal{L}}}
\newcommand{\HH}{{\mathcal{H}}}
\newcommand{\M}{{\mathcal{M}}}

\makeatletter
\def\rightharpoonupfill@{\arrowfill@\relbar\relbar\rightharpoonup}
\newcommand{\xrightharpoonup}[2][]{\ext@arrow
0359\rightharpoonupfill@{#1}{#2}} \makeatother

\newcommand{\res}{\mathop{\hbox{\vrule height 7pt width .5pt depth 0pt
\vrule height .5pt width 6pt depth 0pt}}\nolimits}

\def\e{{\varepsilon}}
\def\O{{\Omega}}

\def\M{{\it M}}

\def\M{{\cal M}}

\def\HH{{\cal H}}

\newtheorem{Theorem}{Theorem}[section]
\newtheorem{Lemma}[Theorem]{Lemma}
\newtheorem{Proposition}[Theorem]{Proposition}

\newtheorem{Remark}[Theorem]{Remark}
\newtheorem{Definition}[Theorem]{Definition}

\newcommand{{\rr}}{{\mathbb R}}

\newenvironment{@abssec}[1]{%
     \if@twocolumn
       \section*{#1}%
     \else
       \vspace{.05in}\footnotesize
       \parindent .2in
         {\upshape\bfseries #1. }\ignorespaces
     \fi}
     {\if@twocolumn\else\par\vspace{.1in}\fi}
\newcommand\keywordsname{Key words}
\newcommand\AMSname{AMS subject classifications}

\begin{document}

\title{Dimension reduction for $-\Delta_1$}
\author{\textsc{Maria Emilia Amendola}\thanks{Dipartimento di Matematica, Universita' degli Studi di Salerno, via Ponte Don Melillo, 84084 Fisciano (SA) Italy. Email: emamendola@unisa.it},
\textsc{Giuliano Gargiulo}, \thanks{DSBGA, Universita' del Sannio, Benevento Italy. Email: ggargiul@unisannio.it}
\textsc{ Elvira Zappale}\thanks{DIEII, Universita'
degli Studi di Salerno, Via Ponte Don Melillo, 84084 Fisciano (SA) Italy.
E-mail:ezappale@unisa.it}}
\maketitle

\begin{abstract}
\noindent A 3D-2D dimension reduction for $-\Delta_1$ is obtained. A power law approximation from $-\Delta_p$ as $p \to 1$ in terms of $\Gamma$- convergence, duality and asymptotics for least gradient functions has also been provided.
 
\medskip

\noindent\textbf{Keywords}: $1$-Laplacian, $\Gamma$-convergence, least gradient functions, dimension reduction, duality.

\noindent\textbf{MSC2010 classification}: 35J92, 49J45, 49K20, 49M29.

\end{abstract}

\section{Introduction}
Recently a great deal of attention has been devoted to thin structures because of the many applications they find in the applied sciences.
A wide literature, concerning mathematical problems defined in thin structures  and modelled through partial differential equations and integral functionals, is available both in the Sobolev and $BV$ settings. To our knowledge little is known when one wants to investigate the relations between problems dealing with thin structures whose deformation fields are functions of bounded variation and the analogous  problems modelled through Sobolev fields. This issue has been in fact pointed out also by \cite{BN},  in the context of applications dealing with approximations of yield sets in Plasticity and for models dealing with dielectric breakdown.  
  
The aim of this paper consists, in fact, in determining the asymptotic behaviour, both for $\e \to 0$ and $p \to 1$ of $p-harmonic$ functions in thin domains of  the type $\Omega_\e: \omega \times \left(-\frac{\e}{2},\frac{\e}{2}\right)$, with prescribed boundary data $u_0$ on the lateral boundary of $\Omega_\e:= \partial \omega \times \left(-\frac{\e}{2},\frac{\e}{2}\right)$, i.e.
\begin{equation}\label{p-e-equation-Omega-e}
\left\{\begin{array}{ll}
-\Delta_p v:=-{\rm div}(|\nabla v|^{p- 2}\nabla v)=0  &\hbox{ in }\Omega_\e,\\
\\
v \equiv v_0 &\hbox{ on } \partial \omega \times \left(-\frac{\e}{2}, \frac{\e}{2}\right),\\
\\
|\nabla v|^{p-2}\nabla v \cdot \nu= 0 &\hbox{ on }  \omega \times\left\{-\frac{\e}{2}, \frac{\e}{2}\right\}, 
\end{array}
\right.
\end{equation} 
where $\nu$ denotes the unit normal to the top and the bottom of the cylinder.

We emphasize the fact that the thin domain is a cylinder, with cross section $\omega$, satisfying suitable regularity requirements, that will be clearly stated in the sequel (see in particular section \ref{LGdimred}). We assume in our subsequent analysis that the boundary is indeed piecewise $C^1$ (see beginning of section \ref{results} . 

Equivalently one may think of studying as $\e \to 0$ and $p \to 1$, the associated Dirichlet integral, namely
\begin{equation}\label{Dirichletintegral-p-e}
\displaystyle{\frac{1}{\e}\int_{\Omega_\e}|\nabla v|^p dx}
\end{equation}
among all the fields $v \in W^{1,p}(\Omega_\e)$, with $v \equiv v_0$ on $\partial \omega \times \left(-\frac{\e}{2}, \frac{\e}{2}\right)$.

Several issues appear at this point, (see for instance \cite{K} for a recent survey on the asymptotics as $p \to 1$): varying domains $\Omega_\e$, meaning of the equation \eqref{p-e-equation-Omega-e} for $p=1$, the possibility and the order with respect to which one may take the limits as $\e \to 0$ and $p \to 1$. The first issue, together with the recalls of the existing literature, is addressed in subsection \ref{RP}. Section \ref{gammaconvergence} is devoted to recall $\Gamma$-convergence, measures, functions of bounded variations, fractional Sobolev spaces, duality theory.  In section \ref{results}, through Proposition \ref{ourresult}, we give a meaning to the anisotropic $1$-Laplacian operator, and provide sufficient conditions ensuring that the limits $p \to 1$ and $\e \to 0$ commute (cf. Theorems \ref{Gammap0to10} and \ref{corollary33AGZAimeta} and Remark \ref{diagram} below). In section \ref{dualityapproach} we make the asymptotics in terms of differential problems via the duality, see Remark \ref{dualityrem} and Proposition \ref{prop1.2KTRobin}. 
Connections with the least gradient problem will be addressed in section \ref{LGdimred}, see Theorems \ref{asymptoticpto1equation} and \ref{lastleast}. This latter approach reveals its importance in determining the existence of solutions to the limit problems (as $p \to 1$)  of \eqref{p-e-equation-Omega-e}. In fact, in spite of lack of coerciveness of Problems \ref{1-e-problemBv} and \ref{1-0-problemBv} below, the solution exists provided suitable geometrical regularity assumptions on the cross section $\omega$ of the cylinder $\Omega_\e$.

\subsection{Rescaling and first results}\label{RP}
We start by rescaling our problem, thus eliminating the varying domains, putting the dependence on $\e$ in the expression of the equation and in its associated variational functional.

\noindent To this end, we fix our notations: let $\omega \subset \mathbb R^2$ be a bounded smooth domain which is piecewise $C^1$ (or whose boundary $\partial \omega$ has positive mean curvature (cf. \cite{SWZ})) and let $u_0 \in X(\partial \omega)$, where $X(\partial \omega)$ denotes a suitable function space to be defined later according to the different formulations of the problems.
For every $\e>0$, let $\Omega_\e$ be a cylindrical domain of cross section $\omega \subset \mathbb R^2$ and thickness $\e$, namely $\Omega_\e:= \omega \times \left(-\frac{\e}{2};\frac{\e}{2}\right)$.  
We reformulate \eqref{Dirichletintegral-p-e}, considering a $\frac{1}{\e}$-dilation in the transverse direction $x_3$.
\begin{equation}\label{changeofvariable}
\begin{array}{ll}
\Omega:=\Omega_1=\omega \times \left(-\frac{1}{2}, \frac{1}{2}\right),\\
\\
u(x_1,x_2,x_3):= v(x_1,x_2,\e x_3), \\
\\
u_0(x_1,x_2)=v_0(x_1,x_2).
\end{array}
\end{equation}

\noindent In the sequel we will denote the planar variables $(x_1,x_2)$ by $x_\alpha$ and for every $\xi_1,\xi_2,\xi_3 \in \mathbb R$, the vector $(\xi_1, \xi_2,\xi_3)$ will be denoted by $\xi_\alpha | \xi_3$. 

\noindent Thus for every $p>1$, \eqref{Dirichletintegral-p-e} is replaced by 
$I_{p,\e}:  W^{1,p}(\Omega) \to \mathbb R^+$, defined as 
\begin{equation}\label{p-e-functional}
\displaystyle{I_{p,\e}(u):= \int_\Omega \left|\nabla_\alpha u \left| \frac{1}{\e}\nabla_3 u\right.\right|^p dx_\alpha dx_3}.
\end{equation}

\noindent We can consider the following variational problem

\begin{equation}\label{p-e-problem}
\displaystyle{{\cal P}_{p,\e}:=\min\left\{ I_{p,\e}(u) : u\in W^{1,p}(\Omega), u \equiv u_0 \hbox{ on }\partial \omega \times \left(-\frac{1}{2}, \frac{1}{2}\right) \right\}.}
\end{equation} 

\noindent The Euler-Lagrange equation associated to \eqref{p-e-problem} is 
\begin{equation}\label{p-e-equation}
\left\{
\begin{array}{ll}
-\Delta_{p,\e}u = 0 &\hbox{ in }\Omega,\\
 \\
u \equiv u_0 &\hbox{ on }\partial \omega\times\left(-\frac{1}{2},\frac{1}{2}\right),
\\
\\
|Id_{\e}\nabla u\cdot \nabla u|^{\frac{p-2}{2}}(Id_{\e}\nabla u)\cdot \nu=0 &\hbox{ on }\omega \times\left\{-\frac{1}{2},\frac{1}{2}\right\},
\end{array}
\right.
\end{equation}
where $Id_\e\in \mathbb R^{3 \times 3}$ is the matrix defined as 
\begin{equation}\label{Ide}
(Id_{\e})_{i,j}= \left\{
\begin{array}{ll}
\frac{1}{\e^2} &\hbox{ if }i=j=3,\\
\\
\delta_{i,j}  &\hbox{ otherwise, }  
\end{array}\right.
\end{equation}

\noindent and $\Delta_{p,\e}$ is the simple anisotropic  ${p,\e}$-Laplace operator defined as
\begin{equation}\nonumber
\Delta_{p,\e} u= {\rm div}\left(\left|Id_\e \nabla u \cdot \nabla u\right|^{\frac{p-2}{2}}\ Id_\e \nabla u \right).
\end{equation} 

\noindent We are interested in the asymptotic behaviour of ${\cal P}_{p,\e}$ and ${\rm argmin}{\cal P}_{p,\e}$, (namely the behaviour of the weak solutions of \eqref{p-e-equation}) both  in the order  ($p \to 1$, $\e \to 0$) and in the reverse one, i.e. ($\e \to 0$, $p \to 1$). 

\noindent In order to exploit pre-existing analysis, we will discuss first the case $\e \to 0$ before $p \to 1$.

\noindent For $\e= 0$ we may introduce the $3D$ problem
\begin{equation}\label{p-0-equation-1}
\left\{
\begin{array}{ll} 
-\Delta_{\alpha, p,0} u :=-{\rm div_\alpha}(|\nabla_\alpha u|^{p-2}\nabla_\alpha u)=0 &\hbox{ in }\Omega,\\
\\
\nabla_3 u = 0&\hbox{ in }\Omega,\\
\\
u= u_0 &\hbox{ in }\partial \omega \times \left(-\frac{1}{2}, \frac{1}{2}\right),
\end{array}
\right.
\end{equation}
\noindent where the index $\alpha$ means that the derivatives are taken just with respect to $x_\alpha$.

Let $I_{p,0}:W^{1,p}(\omega) \to \mathbb R^+$, be the functional defined as
\begin{equation}\label{p-0-functional}
I_{p,0}(u):=\int_\omega \left|\nabla_\alpha u\right|^p dx,
\end{equation}
 and define 
the minimum problem
\begin{equation}\label{p-0-problem}
{\cal P}_{p, 0}:= \min\left\{I_{p,0}(u): u \in W^{1,p}(\Omega), u\equiv u_0 \hbox{ on } \partial \omega \right\}.
\end{equation}

\noindent It is well known since the pioneering papers \cite{ABP} and \cite{LDR} that, for every $p>1$, ${\cal P}_{p,\e}$ converges as $\e \to 0$ to ${\cal P}_{p,0}$, namely the functionals $I_{p,\e}$ $\Gamma$-converge with respect to $L^p$ strong topology, as $\e \to 0$ to $I_{p,0}$, (cf. \cite[Theorem 2]{LDR}). In particular, it has to be observed that the convexity of the space functions in \eqref{p-e-problem} and \eqref{p-0-problem}, the strict convexity and the coerciveness of $I_{p,\e}$ and $I_{p,0}$, due to the choice $p>1$, ensure that ${\cal P}_{p,\e}$ and ${\cal P}_{p,0}$ admit a unique solution, which, in turn is a weak solution of \eqref{p-e-equation} and \eqref{p-0-equation-1}, respectively, for instance when $u_0 \in X(\partial \omega)=W^{\frac{p-1}{p},p}(\partial \omega)$ (cf. subsection \ref{gammaconvergence} for the definition of the fractional Sobolev space). 

At this point it is worth, identifying the fields in $W^{1,p}(\Omega)$ with $\nabla_3 u=0$ with the fields in $W^{1,p}(\omega)$, to observe that \eqref{p-0-equation-1} admits the equivalent $2D$ formulation
\begin{equation}\label{p-0-equation-2}
\left\{
\begin{array}{ll}
-\Delta_{p,0}u:=-{\rm div}(|\nabla u|^{p-2}\nabla u)=0 & \hbox{ in }\omega,\\
\\
u = u_0 &\hbox{ on }\partial \omega. 
\end{array}
\right.
\end{equation}

For every fixed $\e >0$ and $p=1$, one can also define the following variational problems
\begin{equation}\label{1-e-problem}
{\cal P}_{1,\e}:=\inf\left\{I_{1,\e}(u): u \in W^{1,1}(\Omega), u \equiv u_0 \hbox{ on } \partial \omega \times\left(-\frac{1}{2}, \frac{1}{2}\right) \right\},
\end{equation}


\noindent where $I_{1,\e}: W^{1,1}(\Omega) \to \mathbb R^+$, is defined as
\begin{equation}\label{1-e-functional}
I_{1,\e}(u):= \int_\Omega \left|\nabla_\alpha u \left| \frac{1}{\e}\nabla_3 u\right.\right|dx.
\end{equation}
In principle $I_{1,\e}$ may not admit a solution in the Sobolev setting, because of many reasons, first of all the lack of coerciveness, but, as we shall see in section \ref{LGdimred}, also the choice of the space $X(\partial \omega)$, and the regularity of the  set $\Omega_\e$ play a crucial role. 

\noindent Consequently in order to guarantee a correct formulation for problem ${\cal P}_{1,\e}$ one needs to extend $I_{1,\e}$ (with abuse of notations) on the space of functions with bounded variation $BV(\O)$, taking care of the fact that $u=u_0$ outside the lateral boundary of $\Omega$, thus considering
\begin{equation}\label{1-e-functional1}
I_{1,\e}(u):= \left|D_\alpha u\left| \frac{1}{\e} D_3 u\right.\right|(\Omega)
\end{equation}
where the derivatives are intended in the sense of distributions and the integral is replaced by  the total variation. 
Consequently the minimum problem, after a relaxation procedure (cf. \cite[Theorem 3.4]{MZ}) , becomes
\begin{equation}\label{1-e-problemBv}
{\cal P}_{1,\e}=\min\left\{  \left|D_\alpha u\left| \frac{1}{\e} D_3 u\right.\right|(\Omega)+ \int_{\partial \omega \times\left(-\frac{1}{2},\frac{1}{2}\right)} |u-u_0 |d{\cal H}^2, u \in BV(\O)\right\}.
\end{equation}

\noindent Analogously one may consider the problem ${\cal P}_{p,\e}$ for $p=1$ and $\e=0$, thus, formally obtaining
\begin{equation}\label{1-0-problemBv}
{\cal P}_{1,0}=\min \left\{\left|D_\alpha u\right|(\Omega)+ \int_{\partial \omega \times\left(-\frac{1}{2},\frac{1}{2}\right)} |u-u_0 |d{\cal H}^2, u \in BV(\O), D_3u=0\right \},
\end{equation}

\noindent which arises from the relaxation in $BV(\Omega)$ (see \cite{ADM} and \cite{FM}) of the functional  $I_{1,0}: {\cal U} \to \mathbb R$, where ${\cal U}:=\{ u \in W^{1,1}(\Omega) \to \mathbb R^+: \nabla _3 u= 0, u \equiv u_0 \hbox{ on }\partial \omega \times \left(-\frac{1}{2},\frac{1}{2}\right)\}$,  defined as
\begin{equation}\nonumber
I_{1,0}(u):= \int_\Omega \left|\nabla_\alpha u \right|dx,
\end{equation} 
whose related miminum problem in ${\cal U}$ is
\begin{equation}\label{1-0-problem}
{\cal P}_{1,0}:=\inf\left\{I_{1,0}(u): u \in {\cal U} \right\}.
\end{equation}

Also the asymptotic behaviour of $I_{1,\e}$ as $\e \to 0$ is a consequence of the results in \cite{BF}. Namely in \cite[Theorem 1.1]{BF} (see also \cite{BZZ} in presence of bending moments) it has been proven that the almost minimizers of $\{{\cal P}_{1,\e}\}$ in \eqref{1-e-problem}, converge to the solutions (which, in general, may not exist,  cf. section \ref{LGdimred} for sufficient conditions) of ${\cal P}_{1,0}$ in \eqref{1-0-problemBv}.  

Summarizing the above results we can state the following

\begin{Proposition}\label{gammae}
Let $p \geq 1$ and let $u_0 \in X(\partial \omega)=W^{\frac{p-1}{p},p}(\partial \omega)$ for $p >1$ and let $u_0 \in X(\partial \omega)=L^1(\partial \omega)$ for $p=1$.  The families of functionals $\{I_{p,\e}\}$ in \eqref{p-e-functional} and $\{I_{1,\e}\}$ in \eqref{1-e-functional} defined in $\left\{u \in W^{1,p}(\Omega): u \equiv u_0 \right.$ $\left.\hbox{ on }\partial \omega \times\left(-\frac{1}{2},\frac{1}{2}\right)\right\}$, and $\left\{u \in W^{1,1}(\Omega): u \equiv u_0 \hbox{ on }\partial \omega \times\left(-\frac{1}{2},\frac{1}{2}\right)\right\}$, respectively, $\Gamma$-converge as $\e \to 0$, with respect to $L^p$ strong convergence and $L^1$ strong convergence, respectively, to $I_{p,0}$ in \eqref{p-0-functional} and to $\overline{I_{1,0}}(u)= |D u|(\omega)+ \int_{\partial \omega}|u-u_0|d{\cal H}^1$, where this latter functional describes the relaxed functional of $I_{1,0}$ in \eqref{1-0-funcional}, with respect to the $L^1$ strong convergence. 
\end{Proposition}

\begin{Remark}\label{minimizers}
We also recall that given $u_0 \in X(\partial \omega)=W^{\frac{p-1}{p},p}(\partial \omega)$, for $p>1$, the (unique) minimizers of \eqref{p-e-problem} converge, as $\e \to 0$, to the unique element of \eqref{p-0-problem}. For $p=1$ several choices are possible for the boundary datum $u_0$, but in some of these cases, the existence of elements solving \eqref{1-e-problem}, \eqref{1-e-problemBv} and \eqref{1-0-problemBv} is not guaranteed (cf. \cite{SWZ, SZ}), as we will discuss in section \ref{LGdimred}.
\end{Remark}

The asymptotics for $p \to 1$ will be discussed in theorems \ref{Gammap0to10} and \ref{corollary33AGZAimeta}.

\section{$\Gamma$-convergence, measures, functions of bounded variation, trace spaces, and recalls of duality theory}\label{gammaconvergence}
We give a brief survey of $\Gamma$-convergence,  functions of bounded variation and trace spaces. For a detailed treatment of these subjects,
we refer to \cite{DM}, \cite{AFP1}, and \cite{A, B} respectively.

Let $(X,d)$ be a metric space.

\begin{Definition}
[$\Gamma$-convergence for a sequence of functionals]\label{gammasequence} Let
$\{J_{n}\}$ be a sequence of functionals defined on $X$ with values in
$\overline{\mathbb{R}}$. The functional $J:X\rightarrow\overline{\mathbb{R}}$
is said to be the ${\Gamma}-\lim\inf$ (resp. ${\Gamma}-\lim\sup$) of
$\{J_{n}\}$ with respect to the metric $d$ if for every $u\in X$
\[
{J(u)=\inf\left\{  \liminf_{n\rightarrow\infty}J_{n}(u_{n}):u_{n}\in
X,u_{n}\rightarrow u\text{ in }X\right\}  \;\;(}\text{resp. }{\limsup
_{n\rightarrow\infty}).}%
\]
Thus we write
\[
{J=\Gamma-\liminf_{n\rightarrow\infty}J_{n}~}\text{(resp. }{J=\Gamma
-\limsup_{n\rightarrow\infty}J_{n}).}%
\]
Moreover, the functional $J$ is said to be the $\Gamma-$limit of $\{J_{n}\}$
if
\[
{J=\Gamma-\liminf_{n\rightarrow\infty}J_{n}=\Gamma-\limsup_{n\rightarrow
\infty}J_{n},}%
\]
and we may write
\[
{J=\Gamma-\lim_{n\rightarrow\infty}J_{n}.}%
\]

\end{Definition}

For every ${\varepsilon}>0$, let $J_{{\varepsilon}}$ be a functional over $X$
with values in $\overline{\mathbb{R}}$, $J_{\varepsilon}:X \to\overline
{\mathbb{R}}.$

\begin{Definition}
[$\Gamma$-convergence for a family of functionals]\label{gammafamily} A
functional $J:X\rightarrow\overline{\mathbb{R}}$ is said to be the ${\Gamma}%
$-liminf (resp. ${\Gamma}$-limsup or ${\Gamma}$-limit) of $\{J_{\varepsilon
}\}$ with respect to the metric $d$, as ${\varepsilon}\rightarrow0^{+}$, if
for every sequence ${\varepsilon}_{n}\rightarrow0^{+}$
\[
{J=\Gamma-\liminf_{n\rightarrow\infty}J_{{\varepsilon}_{n}}\,(}\text{resp.
}{J=\Gamma-\limsup_{n\rightarrow\infty}J_{{\varepsilon}_{n}}~}\text{or
}{J={\Gamma}-\lim_{n\rightarrow\infty}J_{{\varepsilon}_{n}}),}%
\]
and we write
\[
{J=\Gamma-\liminf_{{\varepsilon}\rightarrow0^{+}}J_{\varepsilon}%
\,(}\text{{resp. }}{J=\Gamma-\limsup_{{\varepsilon}\rightarrow0^{+}%
}J_{{\varepsilon}}}\text{ or }{J={\Gamma}-\lim_{{\varepsilon}\rightarrow0^{+}%
}J_{\varepsilon}).}%
\]

\end{Definition}

Next we state the Urysohn property for $\Gamma$-convergence in a metric space.

\begin{Proposition}
\label{Prop2.3BF}Given $J:X\rightarrow\overline{\mathbb{R}}$ and
$\varepsilon_{n}\rightarrow0^{+},$ $J=\Gamma-\lim\limits_{n\rightarrow\infty
}J_{\varepsilon_{n}}$ if and only if for every subsequence $\left\{
\varepsilon_{n_{j}}\right\}  \equiv\left\{  \varepsilon_{j}\right\}  $ there
exists a further subsequence $\left\{  \varepsilon_{n_{j_{k}}}\right\}
\equiv\left\{  \varepsilon_{k}\right\}  $ such that $\left\{  J_{\varepsilon
_{k}}\right\}  $ $\Gamma-$converges to $J.$
\end{Proposition}

In addition, if the metric space is also separable the following compactness
property holds.

\begin{Proposition}
\label{Thm2.4BF} Each sequence $\varepsilon_{n}\rightarrow0^{+}$ has a
subsequence $\left\{  \varepsilon_{n_{j}}\right\}  \equiv\left\{
\varepsilon_{j}\right\}  $ such that $\Gamma-\lim\limits_{j\rightarrow\infty
}J_{\varepsilon_{j}}$ exists.
\end{Proposition}

\begin{Proposition}
\label{Prop2.5BF} If $J=\Gamma-\underset{\varepsilon\rightarrow0^{+}}{\lim
\inf}J_{\varepsilon}$ (or $\Gamma-\underset{\varepsilon\rightarrow0^{+}%
}{\lim\sup}\,J_{\varepsilon})$ then $J$ is lower semicontinuous (with respect
to the metric $d).$

\end{Proposition}

We conclude with a result dealing with the convergence of minimizers and minimum points, \cite[Corollary 7.17]{DM}
\begin{Theorem}\label{cor7.17DM}
For every $\varepsilon \in \mathbb N$, let $\{x_\varepsilon\}$ a minimizer of $J_\e$ in $X$. If $\{x_\e\}$ converge to $x$ in $X$, then $x$ is a minimizer of $\Gamma-\liminf_ \e J_\e$ and $\Gamma-\limsup_\e J_\e$ in $X$ and 
$$
(\Gamma-\liminf_ \e J_\e) (x)= \liminf_\e J_\e(x_\e), \;\;\;\;\;\;\; (\Gamma-\limsup_ \e J_\e)(x)=\limsup_\e J_\e(x_\e).
$$
\end{Theorem}

We also recall a result that may  be found in \cite{DGL}.

\begin{Proposition}\label{prop28MZ}
Let $O$ be a bounded open set in $\mathbb R^N$, and for every sequence $p>1$, let $\{\mu_p\}_p$ and $\mu$ be non-negative Borel measures on $\Omega$ such that
$$
\left\{
\begin{array}{ll}
\displaystyle{\limsup_{p\to 1}\mu_p(O)\leq \mu(O)< +\infty,} \\
\\
\displaystyle{\limsup_{p \to 1}\mu_p(A)\geq \mu(A) \hbox{ for every open subset }A \hbox{ of }O.}
\end{array}
\right.
$$
Then for every $ \varphi \in C(\overline{O})$ we have
$$
\displaystyle{\lim_{p \to 1} \int_O \varphi d \mu_p = \int_O \varphi d\mu.}
$$ 
\end{Proposition}

Let $O$ be a generic open subset of $\Rb^N$, we denote by
$\M(O)$ the space of all signed Radon measures in $O$ with bounded
total variation. By the Riesz Representation Theorem, $\M(O)$ can
be identified with the dual of the separable space $\C_0(O)$ of
continuous functions on the closure of $O$ vanishing on the boundary $\partial
O$. The $N$-dimensional Lebesgue measure in $\Rb^N$ is designated
as $\LL^N$ while $\HH^{N-1}$ denotes the $(N-1)$-dimensional
Hausdorff measure. If $\mu \in \M(O)$ and $\lambda \in \M(O)$ is a
nonnegative Radon measure, we denote by $\frac{d\mu}{d\lambda}$ the
Radon-Nikod\'ym derivative of $\mu$ with respect to $\lambda$. By a
generalization of the Besicovich Differentiation Theorem (see
\cite[Proposition 2.2]{ADM}), it can be proved that  there exists a
Borel set $E \subset O$ such that $\lambda(E)=0$ and
\begin{equation}\nonumber
\frac{d\mu}{d\lambda}(x)=\lim_{\rho \to 0^+} \frac{\mu(x+\rho \, C)}{\lambda(x+\rho \, C)}\ \text{ for all }x \in {\rm Supp }\,\ \mu \setminus E\end{equation}
and any open convex
set $C$ containing the origin. (Recall that the set $E$ is independent of $C$.)

\vskip5pt

We say that $u \in L^1(O;\Rb^d)$ is a function of bounded
variation, and we write $u \in BV(O;\Rb^d)$, if all its first
distributional derivatives $D_j u_i$ belong to $\M(O)$ for $1\leq i
\leq d$ and $1 \leq j \leq N$. We refer to \cite{AFP1} for a detailed
analysis of $BV$ functions. The matrix-valued measure whose entries
are $D_j u_i$ is denoted by $Du$ and $|Du|$ stands for its total
variation. By the Lebesgue Decomposition Theorem we can split $Du$
into the sum of two mutually singular measures $D^a u$ and $D^s u$
where $D^a u$ is the absolutely continuous part of $Du$ with respect
to the Lebesgue measure $\LL^N$, while $D^s u$ is the singular part
of $Du$ with respect to $\LL^N$. By $\nabla u$ we denote the
Radon-Nikod\'ym derivative of $D^au$ with respect to the Lebesgue
measure so that we can write
$$Du=\nabla u \LL^N + D^s u.$$

The set $S_u$ of points where $u$ does not have an approximate limit is called the approximated discontinuity set, while $J_u \subseteq S_u$ is  the so called jump set of $u$ defined as the set of points $x \in
O$ such that there exist $u^\pm(x) \in \Rb^d$ (with $u^+(x)\neq
u^-(x)$) and $\nu_u(x) \in \mathbb S^{N-1}$ satisfying
$$
\lim_{\varepsilon \to 0}\frac{1}{\varepsilon^N}\int_{\{y\in B_\e(x): (y-x)\cdot \nu_u(x)>0\}} |u(y)-u^+(x)|\,dy=0,
$$
and
$$
\lim_{\varepsilon \to 0}\frac{1}{\varepsilon^N}\int_{\{y\in B_\e(x): (y-x)\cdot \nu_u(x)<0\}} |u(y)-u^-(x)|\,dy=0.
$$

It is known that $J_u$
is a countably $\HH^{N-1}$-rectifiable Borel set.
By the Federer-Vol'pert Theorem (see Theorem 3.78 in \cite{AFP1}), ${\cal H}^{N-1}(S_u \setminus J_u)= 0$ for any $u \in BV(O;\mathbb R^d)$.
The measure $D^s
u$ can in turn be decomposed into the sum of a jump part and a
Cantor part defined by $D^j u:=D^s u \res\, J_u$ and $D^c u:= D^s u
\res\, (O \setminus S_u)$. We now recall the decomposition of $Du$:
$$Du= \nabla u  \LL^N + (u^+ -u^-)\otimes \nu_u {\cal H}^{N-1}\res\,
J_u + D^c u.$$
The three measures above are mutually singular. If ${\cal H}^{N-1}(B) < + \infty$, then $|D^c u|(B)=0$ and there exists a Borel set $E$ such that
$$
{\cal L}^N(E)= 0, \; |D^c u|(X)=|D^c u|(X \cap E)
$$
for all Borel sets $X \subseteq O$.

If $O$ is an open set with Lipschitz boundary $\partial O$ and $u \in BV(O)$, we denote by $u_o$ the null extension of $u$ to $\mathbb R^N$ defined by 
\begin{equation}\nonumber
\left\{
\begin{array}{ll}
u(x) & \hbox{ if } x \in O,
\\
0 &\hbox{ if } x \in \mathbb R^N \setminus O,
\end{array}
\right.
\end{equation}
for ${\cal L}^N$ a.e. $x \in \mathbb R^N$. It turns out that $u_o\in BV(\mathbb R^N)$, and we define the trace $\gamma_O(u)$ of $u$ on $\partial O$ as 
\begin{equation}\nonumber
\gamma_O(u)=(u_o)^+ -(u_o)^-.
\end{equation} 
It results that for ${\cal H}^{N-1}$-a.e. $x \in \partial O$, the vector $\nu_{u_o}(x)$ agrees with the exterior (interior) normal ${\bf n}(x)$ to $\partial O$ at $x$, moreover $u_o^+(x)=0$ or $u_o^-(x)=0$  and $\gamma_O(u)(x)=u^+_o$ or $\gamma_O(u)(x)= u^-_o$.
We also recall that (see \cite{Z})
\begin{equation}\nonumber
\displaystyle{\lim_{r^N\to 0} \frac{1}{r^N} \int_{O\cap B_r(x_0)} |u(x)-\gamma_O(u)(x_0)|^{\frac{N}{N-1}}dx=0 \hbox{  for }{\cal H}^{N-1}-\hbox{a.e.} x_0 \in \partial O.}
\end{equation}

Let $O\subset \mathbb R^N$ be a bounded open set with Lipschitz boudary, $p>1$, the fractional Sobolev space $W^{1-\frac{1}{p},p}(\partial O)$ may be defined as follows.
\begin{equation}\nonumber
\displaystyle{W^{1-\frac{1}{p},p}(\partial O)=\left\{ \varphi \in L^p(\partial O):\int_{\partial O}\int_{\partial O}\frac{|\varphi(x)-\varphi(y)|^p}{|x-y|^{N+p-1}}d {\cal H}^{N-1}d {\cal H}^{N-1}< + \infty \right\}},
\end{equation}
it is endowed with the norm
\begin{equation}\nonumber
\displaystyle{\|\varphi\|_{W^{1-\frac{1}{p},p}(\partial O)}=\left\{\|\varphi\|^p_{L^p(\partial O)}+\int_{\partial O}\int_{\partial O}\frac{|\varphi(x)-\varphi(y)|^p}{|x-y|^{N+p-1}}d {\cal H}^{N-1}d {\cal H}^{N-1} \right\}^\frac{1}{p}}.
\end{equation}

It is well known that $W^{1-\frac{1}{p},p}(\partial O)$ is the trace space of $W^{1,p}(O)$, (i.e. $W^{1-\frac{1}{p},p}(\partial O)= \gamma_O(W^{1,p}(O))$). For $p=1$, one may substitute $W^{1-\frac{1}{p},p}(\partial O)$ by $L^1(\partial O)$. 
Since obviously 
$$
\gamma_O(u)(x)=u(x)
$$
for every $u \in W^{1,p}(O)\cap C(\overline{O})$ and for ${\cal H}^{N-1}$-a.e. $x \in \partial O$, then, with an abuse of notations, in the sequel, we will denote $\gamma_O(u)$ by $u$.
It verifies
$$
\displaystyle{\int_O u {\rm div}\phi dx=-\int_O \nabla u \cdot \phi dx + \int_{\partial O}\phi \gamma_O(u)\cdot {\bf n}_O d {\cal H}^{N-1},}
$$
for every $u \in W^{1,p}(O), \phi \in C^1_c(\mathbb R^N)^N.$

The following inequalities hold
\begin{equation}\label{133Z}
\displaystyle{\|\gamma_O(u)\|_{W^{1-\frac{1}{p},p}(\partial O)} \leq C_0 \|u\|_{W^{1,p}(O)} \hbox{ for every }u \in W^{1,p}(O),}
\end{equation}
and, conversely, for every $\varphi \in W^{1-\frac{1}{p},p}(\partial O)$ there exists $u \in W^{1,p}(O)$ such that $\gamma_O(u)=\varphi$ and
\begin{equation}\label{134Z}
\displaystyle{\|u\|_{W^{1,p}(O)}\leq C_1 \|\varphi\|_{W^{1-\frac{1}{p},p}(\partial O)},}
\end{equation} 
for suitable constants $C_0, C_1 \geq 0$.

The following result (cf. \cite[Proposition 1.1]{Za}) allows us to extend the previous considerations and inequality \eqref{133Z} to $\mathbb R^N \setminus \overline{O}$, provided $O$ is bounded.

\begin{Proposition}\label{propZa}
Let $p>1$, let $O$ be a bounded open set with Lipschitz boundary, then there exists $C'_2>0$ such that for every $\varphi \in W^{1-\frac{1}{p}, p}(\partial O)$ there exists $u \in W^{1,p}(\mathbb R^N \setminus \overline{O})$ such that $\gamma_{\mathbb R^N \setminus \overline{O}}(u)=\varphi$ and 
$$
\displaystyle{\|u\|_{W^{1,p}(\mathbb R^N \setminus \overline{O})} \leq C'_2 \|\varphi\|_{W^{1-\frac{1}{p}, p}(\partial O)}}.
$$
\end{Proposition}

For every $p \in [1,+\infty[$, let $I$ be a bounded open set in $\mathbb R^N$ with Lipschitz boundary such that  $\Gamma:= \partial O \cap I \not= \emptyset$ and suppose that ${\cal H}^{N-1}(\overline{\Gamma}\setminus \Gamma)= 0$, we denote by $W_{0,\Gamma}^{1,p}(O)$ the space
$\{u \in W^{1,p}(O): u = 0 \, {\cal H}^{N-1}-\hbox{a. e. on $\Gamma$}\}$, $W^{1,p}_{0, \partial O}(O)= W^{1,p}_0(O)$. 
In the sequel, for every $u_1 \in W^{1,p}_{\rm loc}(\mathbb R^N)$ we denote $u_1+ W^{1,p}_{0, \Gamma}(O)$ by $W^{1,p}_{u_1, \Gamma}(O)$, and $u_1 + W^{1,p}_{0}(O)$ by $W^{1,p}_{u_1}(O)$.


In the sequel with an abuse of notation, we will identify (the restriction of) a function $u$ with its trace on part of $\partial O$, $\gamma_O(u)$.  

We end this subsection by recalling a result due to Ekeland and Temam that will be exploited in the sequel, we refer to the version mentioned in \cite[Theorem 2]{D3}.

\begin{Theorem}\label{Theorem ET} Suppose that $X$ and $Y$ are Banach spaces, that $\Lambda$ is a linear and continuous operator which sends $X$ into $Y$, that $F$ and $G$ are convex functions on $X$ and $Y$, respectively. We denote $F^\ast$ and $G^\ast$  their Fenchel conjugate, defined, respectively, on $X^\ast$ and $Y^\ast=Y$, by $\Lambda^\ast$  the adjoint operator of $\Lambda$. Then
$$
\inf_{u \in X}\left\{F(u)+ G(\Lambda u)\right\} \geq \sup_{p^\ast \in Y}\left\{-F^\ast(\Lambda^\ast p^\ast)- G^\ast(-p^\ast)\right\}.
$$ 

Suppose that there exists $u_0 \in X$, such that $F(u_0)< \infty$, and $G$ is continuous on $\Lambda u_0$. Then ,
$$
\inf_{u \in X}\left\{F(u)+ G(\Lambda u)\right\} =\sup_{p^\ast \in Y}\left\{-F^\ast(\Lambda^\ast p^\ast)- G^\ast(-p^\ast)\right\},
$$
and the dual problem on the right-hand side of the above possesses at least one solution.
\end{Theorem}

\section{Asymptotics in terms of $\Gamma$-convergence}\label{results}

In order to study the asymptotics for $p \to 1$ of problems ${\cal P}_{p,\e}$ and ${\cal P}_{p,0}$ in \eqref{p-e-problem} and \eqref{p-0-problem} respectively, we will exploit previous results and prove more general ones for generic open sets $O \subset \mathbb R^N$. Finally we will apply these lemmata to the specific open sets $\Omega \subset \mathbb R^3$ and $\omega \in \mathbb R^2$ involved in problems ${\cal P}_{p, 0}$ and ${\cal P}_{p,\e}$. We assume from now on that $\omega$ is a bounded open set in $\mathbb R^2$, which is piecewise $C^1$.
We conjecture that it is possible to assume $\omega$ with Lipschitz boundary, but, since our aim consists of providing $\Gamma$-convergence results in dimension reduction for $-\Delta_1$, connecting our results with `Least Gradient' theory, we did not focus on the regularity assumptions
for the boundary $\partial \omega$.

We start by recalling the following result that can be found in \cite{D} and \cite{G}.

\begin{Proposition}\label{prop2.3Aimeta}
Let $O \subset \mathbb R^N$ be some bounded open set, which is piecewise $C^1$. Let $u_1 \in L^1(\partial O)$. Suppose that $u_p \in W^{\frac{p-1}{p},p}(\partial O)$ converges in $L^1(\partial O)$ towards $u_1$. Then for every $u \in BV(O)$, there exists $U_p \in W^{1,p}(O)$, $U_p=u_p$ on $\partial O$, such that
$$
\begin{array}{ll}
\displaystyle{\lim_{p\to 1}\int_O |\nabla U_p|^p dx= |Du |(O)+ \int_O |u-u_1|d{\cal H}^{N-1},}\\
\\
\displaystyle{\lim_{p\to 1}\int_O |U_p-u|^{1^\ast} dx =0,}
\end{array}
$$
where $1^\ast= \frac{N}{N-1}$.
\end{Proposition} 

We restate the above result in terms of $\Gamma$-convergence with respect to $L^1$-strong convergence. To this end, assume that 
${\bar p}>1$ and let $u_0 \in X(\partial \omega)=W^{\frac{{\bar p}-1}{\bar p}, {\bar p}}(\partial \omega)$. By virtue of Proposition \ref{propZa} and \eqref{134Z}, $u_0$ can be seen as a function in $W^{1,\overline{p}}_{\rm loc}(\mathbb R^2)$ and since it is independent on $x_3$, it can be also considered as a function in $W^{1,{\overline p}}_{\rm loc}(\mathbb R^3)$.
  
Let $F_{p,0}:BV(\omega)\to \mathbb R$ be the functional defined as
\begin{equation}\label{19AGZAIMETA}
F_{p,0}(u):=\left\{
\begin{array}{ll}
\displaystyle{\left(\int_\omega |\nabla u|^p dx\right)^{\frac{1}{p}}} &\hbox{ if }u \in W^{1,p}_{u_0}(\omega),\\
\\
+\infty &\hbox{otherwise}.
\end{array}
\right.
\end{equation}
 Let $F_{1,0}:BV(\omega )\to \mathbb R$ be defined as
\begin{equation}\label{20AGZAIMETA}
F_{1,0}(u):=|D u|(\overline{\omega})=|D u|(\omega)+\int_{\partial \omega} |u-u_0|d{\cal H}^1.
\end{equation}

We can prove the following theorem 
\begin{Theorem}\label{Gammap0to10}
Let $\{F_{p,0}\}_p$ be the family of functionals introduced in \eqref{19AGZAIMETA}, then $\{F_{p,0}\}_p$ $\Gamma$-converges with respect to the $L^1(\omega)$ strong topology to $F_{1,0}$.
\end{Theorem}
\begin{proof}[Proof]
The lower bound is trivially obtained if $\{u_p\}_p$ is such that $\lim_{p \to 1}F_{p,0}(u_p)=+ \infty$.
Let $\{u_p\}_p$ be converging strongly to $u \in BV(\omega)$ in $L^1(\omega)$ and assume also that it is an equibounded energy sequence, namely there exists $C>0$ such that 
$$
\displaystyle{F_{p,0}(u_p)=\left(\int_\omega |\nabla u_p|^p dx\right)^{\frac{1}{p}}\leq C.}
$$ 
By Hoelder inequality, and the fact that $u_p \in W^{1,p}_{u_0}(\omega)$ it results that
$$
\displaystyle{|Du_p|(\overline{\omega})\leq \left(\int_\omega |\nabla u_p|^p dx\right)^\frac{1}{p}|\omega|^{1-\frac{1}{p}}\leq C' \hbox{ for every } 1 \leq p\leq {\bar p}.}
$$

Observe that for every $v \in BV(\omega)$ by virtue of Poincar\'e inequality, any energy equibounded sequence $\{u_p\}_p$ admits a further subsequence, converging weakly $\ast$ in $BV(\omega)$ to $u \in BV(\omega)$. Since $u_0$ is in $W^{\frac{{\bar p}-1}{{\bar p}},{\bar p}}(\partial \omega)$ for a certain ${\bar p}>1$, by virtue of Proposition \ref{propZa}, we can define a function $\widetilde{u_0}$ in $\mathbb R^2 \setminus \overline{\omega}$, whose trace on $\partial \omega$ is $u_0$, and such that $\widetilde{u_0} \in W^{1,\overline{p}}(\mathbb R^2 \setminus \overline{\omega})$.  Again, the regularity assumptions on $\partial \omega$ ensure that we can extend $u\in BV(\omega)$ as $\widetilde{u_0}$ in $\mathbb R^2\setminus \overline{\omega}$, thus obtaining a $BV(\mathbb R^2)$ function, still denoted by $u$.
In the same way we may extend, with an abuse of notations, any  $u_p$, as $\widetilde{u_0} \in \mathbb R^2 \setminus \overline{\omega}$, getting $u_p \in W^{1,p}(\mathbb R^2)$.

Clearly the functions, extended as above, are such that $\{u_p\}_p$ weakly $\ast$  converge to $u $ in $BV(\omega')$ for any bounded open set $\omega' \supset \supset \omega$.
Consequently the lower semicontinuity of the total variation with respect to the weak $\ast$ topology in $BV$, and Hoelder inequality provide the following chain of inequalities
$$
\begin{array}{ll}
\displaystyle{|D u|(\omega') \leq \liminf_{p \to 1}|D u_p|(\omega ')\leq \liminf_{p \to 1} \left(\int_{\omega'} |\nabla u_p|^p dx\right)^\frac{1}{p}|\omega '|^{1-\frac{1}{p}}=}\\
\displaystyle{=\liminf_{p \to 1}\left(\int_\omega|\nabla u_p|^p dx + \int_{\omega' \setminus \overline{\omega}}|\nabla \widetilde{u_0}|^pdx\right)^{\frac{1}{p}}, \hbox{ for every } p \leq \overline{p}. }
\end{array}
$$

As $\omega'$ shrinks to $\omega$, by \eqref{20AGZAIMETA}, we obtain the so called $\Gamma$-liminf inequality
$$
\displaystyle{|Du|(\overline{\omega}) \leq \liminf_{p \to 1}\left(\int_\omega|\nabla u_p|^p dx\right)^{\frac{1}{p}},  \hbox{ for every } p \leq \overline{p}.}
$$
For what concerns the upper bound, we invoke Proposition \ref{prop2.3Aimeta}, thus for every $u \in BV(\omega)$ we get the existence of a sequence $\{u_p\}_p\in W^{1,p}_{u_0}(\omega)$ such that
$$
\begin{array}{ll}
\displaystyle{\lim_{p \to 1}\int_\omega |u_p- u|^{1 \ast}dx=0,}\\
\\
\displaystyle{\lim_{p \to 1}\left(\int_\omega |\nabla u_p|^p dx \right)^{\frac{1}{p}}=|D u|(\omega)+ \int_{\partial \omega} |u-u_0|d{\cal H}^1},
\end{array}
$$
and that concludes the proof.
\end{proof}

The following result carries Proposition \ref{prop2.3Aimeta} to more general integrals.

\begin{Proposition}\label{prop3.2Aimeta}
Let $O \subset \mathbb R^N$  be some bounded open set, piecewise $C^1$. Let $W:\mathbb R^N \to [0, +\infty[$ be a continuous, positively $1$-homogeneous function such that
\begin{equation}\label{Wgrowth}
\frac{1}{C}|\xi|\leq W(\xi)\leq C|\xi| \hbox{ for every }\xi \in \mathbb R^N,
\end{equation}
for a suitable positive constant $C$. Let $u_1\in W^{1-\frac{1}{{\bar p}},{\bar p}}(\partial O)$, for some ${\bar p} >1$. Then, for every $u \in BV(O)$, and for every $1<p \leq \overline{p}$, there exists $U_p \in W^{1,p}(O)$, $U_p=u_1$ on $\partial O$, such that
$$
\begin{array}{ll}
\displaystyle{\lim_{p \to 1}\int_O (W(\nabla U_p))^p dx = \int_O W\left(\frac{d D u}{d |D u|}\right)d |D u|+ \int_{\partial O}W((u-u_1){\bf n})d{\cal H}^{N-1}},\\
\displaystyle{\lim_{p \to 1}\int_O |U_p-u|^{1^\ast}dx =0,}
\end{array}
$$ 
where ${\bf n}$ is the unit exterior normal to $\partial O$, and $1^\ast =\frac{N}{N-1}$.
\end{Proposition}
\begin{proof}[Proof]
Let $u \in BV(O)$, first  we claim that for every sequence $\{p\}$ converging to $1$, with $p \geq 1$, it is possible to find  a subsequence, still denoted  by $\{p\}$ and a sequence $\{v_p\}\subset W^{1,p}(O)\cap C^\infty(O)$, with $v_p=u_1$ on $\partial O$ such that
\begin{equation}\label{limit1ast}
\displaystyle{\lim_{p \to 1}\int_O|v_p - u|^{1^\ast}dx =0}
\end{equation} 
and
\begin{equation}\label{Resh}
\displaystyle{\lim_{p \to 1}\int_O W(\nabla v_p)dx= \int_O W\left(\frac{d D u}{d |Du|}\right)d |D u|+ \int_{\partial O}W((u-u_1){\bf n})d {\cal H}^{N-1}.}
\end{equation}
To prove the claim we observe that \cite[Proposition 2]{D2} ensures that there exists a sequence $\{v_p\}_p$ such that $v_p \in W^{1,p}(O)\cap C^\infty(O)$, and $v_p = u_1$ on $\partial O$, \eqref{limit1ast} holds, 
$\displaystyle{\lim_{p \to 1} \int_O |v_p- u|^{1^\ast }dx=0}$ and $\displaystyle{\lim_{p \to 1}\int_O |\nabla v_p|^p dx=}$ $\displaystyle{ |D u|(O)+ \int_{\partial O}|u-u_1|d {\cal H}^{N-1}.}$
This in turn, by virtue of Hoelder inequality, implies that 
\noindent $ \displaystyle{\lim_{p \to 1}\int_O |\nabla v_p|dx \leq}$ $\displaystyle{ \lim_{p \to 1}\left(\int_O |\nabla v_p|^pdx\right)^\frac{1}{p}|O|^{1-\frac{1}{p}}}$ $\displaystyle{=|D u|(O)+ \int_{\partial O}|u-u_1|d {\cal H}^{N-1} .}$

\noindent The opposite inequality follows by well known relaxation results, see \cite{GS}, where the functional $|D v|(O)+ \int_{\partial O}|v-u_1|d {\cal H}^{N-1}$ turns out to be the relaxed functional (with respect to $L^1(O)$ strong convergence) of $\left\{
\begin{array}{ll}\int_O |\nabla v|dx & \hbox{ if }v \in W^{1,1}_{u_1}(O), \\
+\infty & \hbox{ if }v \in BV(O) \setminus W^{1,1}_{u_1}(O).
\end{array}
\right.  
 $ 

\noindent Now, observing that $u_1\in W^{1-\frac{1}{p},p}(\partial O)$ can be extended as a $W^{1,p}(\mathbb R^N\setminus O)$ function, still denoted by $u_1$ outside $O$ (see Proposition \ref{propZa}),  we can extend $v_p$ and $u$ as $u_1$ outside $O$, thus obtaining a $W^{1,p}(\mathbb R^N)$ function and a $BV(\mathbb R^N)$ one (see the end of subsection \ref{gammaconvergence}  and \cite[Corollary 3.89]{AFP1}), respectively.
Consequently for every open set $O'\supset \supset O$, applying Reshetnyak's continuity theorem \cite[Theorem 2.39]{AFP1}, it results
$$
\begin{array}{ll}
\displaystyle{\lim_{p \to 1}\int_{O'}|\nabla v_p|dx=\lim_{p\to 1} \left(\int_{O'\setminus O}|\nabla u_1|dx+ \int_O|\nabla v_p|dx\right) =}
\displaystyle{|Du|(O)+ \int_{\partial O}|u-u_1|d{\cal H}^{N-1}+\int_{O'\setminus O}|\nabla u_1|dx.}
\end{array}
$$   
Thus, as $O'$ shrinks to $O$, we obtain \eqref{Resh} and this proves the claim.


Next, the density of smooth functions in $W^{1,p}(O)$, with respect to strong $W^{1,p}$ convergence, the Sobolev embedding theorems and the continuity of $W$ imply that there exists a further sequence $\{w_{q}\}_{q} \in W^{1,p}(O)\cap C^\infty(O)$, with $w_{q}\equiv u_1$ on $\partial O$,  converging strongly in $W^{1,p}(O)$ to $v_p$ as $q\to 1$, such that $\displaystyle{\lim_{q\to 1}\int_O |v_p-w_q|^{1^\ast}dx =0}$, $\nabla w_{q}$ and $W^q(\nabla w_{q})$ pointwise converge a.e. to $\nabla v_p$ and $W(\nabla v_p)$ respectively, as $q \to 1$.

The growth from above in \eqref{Wgrowth}, and Hoelder inequality entail that $ W^q(\nabla w_{q})$ is equintegrable, thus we can conclude that $\int_O W^q(\nabla w_{q})dx$ converges to $\int_O W(\nabla v_p)dx$ as $q \to 1$. 

Finally a diagonal argument guarantees that there exists another sequence in $W^{1,p}(O)\cap C^\infty(O)$, denoted by $\{U_p\}$ such that $U_p\equiv u_1$ on $\partial O$, \eqref{limit1ast} holds and
\begin{equation}\label{due}
\displaystyle{\lim_{p \to 1}\int_O W^p(\nabla U_p) dx =\int_O W\left(\frac{d DU}{d|DU|}\right) d |DU|+ \int_{\partial O}W((U-u_1){\bf n})d {\cal H}^{N-1}.}
\end{equation}

Since the above arguments can be repeated, extracting, with an abuse of notation, a subsequence $\{p\}$ and a corresponding sequence $\{U_p\}$, verifying \eqref{limit1ast} and \eqref{due} from any given $\{p\}$ and $\{v_p\}$, we can conclude that this construction is possible for any $p \to 1$ and this concludes the proof.
 
\end{proof}

Let $ \omega \subset \mathbb R^2$ be a bounded open set, piecewise $C^1$ and let  $H_{p,0}:BV(\omega)\to \mathbb R$ be the family of functionals defined as
\begin{equation}\label{Hp0}
H_{p,0}(u):= \left\{
\begin{array}{ll}
\displaystyle{\left(\int_\omega W^p(\nabla u)dx \right)^{\frac{1}{p}}} &\hbox{ if } u \in W^{1,p}_{u_0}(\omega),\\
\\
+\infty &\hbox{otherwise,}
\end{array}
\right.
\end{equation}
where $W:\mathbb R^N \to [0,+\infty[$ is  convex, positively $1$-homogeneous and verifying \eqref{Wgrowth}.
\begin{Theorem}\label{GammaWp0to10}
The family of functionals $\{H_{p,0}\}_p$ defined in \eqref{Hp0}, $\Gamma$-converges, as $p $ tends to $1$ and with respect to $L^1$ strong convergence, towards the functional $H_{1,0}:BV(\omega)\to \mathbb R$ defined as 
\begin{equation}\label{H10}
\displaystyle{H_{1,0}(u):=\int_\omega W\left(\frac{d Du}{d|Du|}\right)d |D u|+ \int_{\partial \omega}W((u_0-u)\nu) d {\cal H}^{N-1},}
\end{equation}
where $\nu$ denotes the unit exterior normal to $\partial \omega$.
\end{Theorem}
\begin{proof}[Proof]
The proof develops along the same lines as Theorem \ref{Gammap0to10}. Namely the lower bound can be proved arguing exactly as in the latter theorem, just exploiting the lower semicontinuity with respect to $BV$-weak $\ast$ convergence, of the functional $H_{1,0}$ as proven in \cite{GMS}.  
On the other hand the upper bound is immediate consequence of Proposition \ref{prop3.2Aimeta}. 
\end{proof}
Let ${\overline p} >1$ and let $u_0 \in W^{1-\frac{1}{\overline p},{\overline p}}(\partial \omega)$. Clearly the regularity of $\omega$,  \eqref{134Z} and Proposition \ref{propZa} ensure that $u_0$ can be naturally extended to a function in $W^{1,p}_{\rm loc}(\mathbb R^2)$, in turn with an abuse of notations, this latter function can be regarded as a function depending also on $x_3$, $u \in W^{1,p}_{\rm loc}(\mathbb R^3)$.

Having in mind the functionals $\{I_{p,\e}\}_{p,\e}$ quoted in \eqref{p-e-functional} in the Introduction, we define, for every $p>1$ and $\e>0$, $F_{p,\e}:BV(\Omega)\to \mathbb R$ as the functionals
\begin{equation}\label{Fpe}
F_{p,\e}(u):=\left\{
\begin{array}{ll}
\displaystyle{\left(\int_{\O}\left|\nabla_\alpha u\Big| \frac{1}{\e}\nabla_3 u \right|^p dx\right)^{\frac{1}{p}}} &\hbox{ if }u \in W^{1,p}_{u_0, {\rm lat}}(\O),
\\
\\
+ \infty &\hbox{ otherwise,}
\end{array}
\right.
\end{equation}

\noindent where $W^{1,p}_{u_0, {\rm lat}}(\Omega) =\left\{u \in W^{1,p}(\Omega): \right.$ 
$\left. u \equiv u_0 \;{\cal H}^2- \hbox{a.e. on } \partial \omega \times \left(-\frac{1}{2}, \frac{1}{2}\right)\right\},$
(cf. Subsection \ref{gammaconvergence} and observe that ${\cal H}^2\left(\overline{\omega \times \left\{-\frac{1}{2}, \frac{1}{2}\right\}}\setminus \omega \times \left\{-\frac{1}{2}, \frac{1}{2}\right\} \right)=0$).


Analogously, for $p=1$, let $W_\e:\mathbb R^3 \to \mathbb R$ be the function defined as $W_\e(\xi)=W_\e(\xi_1,\xi_2,\xi_3)=\left|\xi_\alpha|\frac{1}{\e}\xi_3\right|$, we can recall the functionals $I_{1,\e}:BV(\Omega)\to \mathbb R$, introduced in \eqref{1-e-functional1}, as
\begin{equation}\label{F1e}
\displaystyle{I_{1,\e}(u):=\left|D_\alpha u\Big|  \frac{1}{\e}D_3 u\right|(\Omega)+ \int_{\partial \omega \times \left(-\frac{1}{2},\frac{1}{2}\right)} W_\e((u-u_0)\nu)d{\cal H}^2},
\end{equation}
where $\nu$ is the unit vector perpendicular to $\partial \omega \times \left(-\frac{1}{2},\frac{1}{2}\right)$.

We observe that the restriction of $I_{1,\e}$ to $W^{1,1}_{u_0, {\rm lat}}(\Omega)$   is given by \eqref{1-e-functional}.

\noindent Moreover, for every $\e >0$, let $G_{1,\e}:BV(\Omega)\to [0,+\infty)$ be the functionals defined as 
\begin{equation}\label{G1e}
G_{1,\e}(u):=\left\{
\begin{array}{ll}
\int_\Omega \left|\nabla_\alpha u \Big| \frac{1}{\e}\nabla_3 u\right|dx &\hbox{ if } u \in W^{1,1}_{u_0, {\rm lat}}(\Omega),\\
\\
+\infty &\hbox{ otherwise.}
\end{array}
\right.
\end{equation}

\noindent Then, their relaxed functionals (with respect to $L^1$- strong topology) coincide with the functionals $I_{1,\e}$ in \eqref{F1e} (cf. \cite[Theorem 3.4]{MZ}).

To prove the $\Gamma$-convergence of $\{F_{p,\e}\}_p$ towards $I_{1,\e}$ in \eqref{F1e} as $p \to 1$ we need some preliminary results in the same spirit of those proposed in \cite{MZ}, which need the assumption ${\cal H}^{2}(\overline{\omega \times \left\{-\frac{1}{2},\frac{1}{2}\right\}}\setminus \omega \times \left\{-\frac{1}{2},\frac{1}{2}\right\})=0$. 
We also observe that, having in mind the subsequent applications to $-\Delta_1$- type equations, and for the sake of simplicity in the exposition of the proof, we consider an energy density $W$ positively $1$-homogeneous, but analogous results hold replacing $W$ with its recession function $W^\infty$ where necessary.


\begin{Lemma}\label{prop31MZ}
Let $\omega$ be a bounded open set in $\mathbb R^2$, piecewise $C^1$, and $\Omega =\omega \times \left(-\frac{1}{2},\frac{1}{2}\right)$,  let $W:\mathbb R^3 \to [0,+\infty[$ be a convex and positively $1$-homogeneous function, satisfying \eqref{Wgrowth}. Let $1<  {\bar p}$ and let $u_0 \in W^{\frac{{\bar p -1}}{\bar p}, {\bar p}}(\partial \omega)$, then for every $u \in BV(\Omega)$, for every $1<p \leq {\bar p}$ it results
\begin{equation}\label{lowerboundp1}
\displaystyle{\int_\Omega W\left(\frac{d Du}{d |D u|}\right)d |D u| + \int_{\partial \omega \times \left(-\frac{1}{2},\frac{1}{2}\right)} W((u_0-u)\nu)d{\cal H}^2 \leq \liminf_{p \to 1}\int_\Omega W^p(\nabla u_p)dx}
\end{equation}
for every sequence $\{u_p\}_p$ with $u_p\in W^{1,p}_{u_0, {\rm lat}(\Omega)}$, such that $u_p \to u$ in $L^1(\Omega)$.
\end{Lemma}
\begin{proof}[Proof]
The result easily follows from the lower semicontinuity with respect to $L^1(\Omega)$ strong topology of the left hand side of \eqref{lowerboundp1} as proven in \cite[Proposition 3.1]{MZ} and the Hoelder inequality.
\end{proof}

Now we introduce the following notations, already adopted in \cite{DAT, MZ}.
We say that an open set $O\subset \mathbb R^N$ is cone-shaped if and only if there exists $x_0 \in \mathbb R^N$, $S \subset \mathbb R^N$, such that 
$$
\displaystyle{O=\{(1-t)x_0 +t x: x \in S, t \in ]0,1[\}}.
$$ 
We call $x_0$ the vertex of $O$, $S$ the basis of $O$ and observe that, if $t \in ]0,1[$, then
$$
\displaystyle{x_0 +t (O- x_0)\subset O, \;\;\;\; x_0 + t(S-x_0)\subset O.}
$$
Let $x_0 \in \mathbb R^N$ and $S \subset \mathbb R^N$ we denote by $C_{x_0,S}$ the cone
$$
\displaystyle{C_{x_0,S}=\{(1-t)x_0+ t x: x \in S, t >0\}.}
$$
 In what follows we will consider cone-shaped sets of vertex $x_0$ and basis $S$ such that for any fixed $x \in S$, one has,
\begin{equation}\label{212MZ}
\displaystyle{\{(1-t)x_0+ t x: t \in [0,1]\}\cap S = \{x\}.}
\end{equation}

The following lemma develops along the lines of \cite[Lemma 2.1]{DAT}.

\begin{Lemma}\label{comelemma21DAT}
Let $1 < \overline{p}$ and let $W:\mathbb R^N \to [0,+\infty[$ be convex,  positively $1$-homogeneous and verify \eqref{Wgrowth}. Let $u_1\in W^{1,\overline{ p}}_{\rm loc}(\mathbb R^N)$, $O$ an open set piecewise $C^1$. Let $A,B$ be open sets such that $A \subseteq O, A \subset \subset B, O \setminus \overline{B}\not = \emptyset$, and let us assume that $O\cap B$ is piecewise $C^1$. Let $u \in BV(O)$, with $u=u_1$ a.e. in $O \setminus A$, then there exists $\{u_p\}_p$ such that $u_p \in W^{1,p}_{\rm loc}(\mathbb R^N)$ and $u_p \equiv u_1$ a.e. in $O\setminus B$ for every $1<p\leq {\bar p}$ and 
$$
\displaystyle{\lim_{p \to 1}\int_O |u_p-u|^{1^\ast}dx =0,}
$$
and
$$
\begin{array}{ll}
\displaystyle{\lim_{p \to 1}\int_O W^p(\nabla u_p) dx \leq \int_{O \cap B}W(\nabla u)dx +}\\
\\
\displaystyle{\int_A W\left(\frac{d D^s u}{d |D^s u|}\right)d |D^s u|+\int_{O \cap \partial A}W((u-u_1){\bf n})d{\cal H}^{N-1} +\int_{O \setminus A}W(\nabla u_1)dx.}
\end{array}
$$
\end{Lemma}
\begin{proof}[Proof]
Since $O \cap B$ has Lipschitz boundary, by virtue of Proposition \ref{prop3.2Aimeta}, applied to $O \cap B$, and since $u\equiv u_1$ a.e. in $O \setminus A$ we know that  for every $\overline{p} \geq p>1$ there exists $\{v_p\}_p$ with $v_p  \in W^{1,p}_{\rm loc}(\mathbb R^N)$ , such that
$$
\displaystyle{\lim_{p\to 1}\int_{O\cap B}|v_p-u|^{1^\ast}dx =0,}
$$
\noindent and
\begin{equation}\label{22DAT}
\displaystyle{\lim_{p \to 1}\int_{O \cap B} (W(\nabla v_p))^pdx= \int_{O \cap B}W\left(\frac{d D u}{d |D u|}\right)d |D u|.}
 \end{equation}
For every  sequence $p >1$, $k \in \mathbb N$, let $\chi_k:\mathbb R \to \mathbb R$ be a smooth function such that $0 \leq \chi'_k\leq 1$ with 
$$
\chi_k=\left\{
\begin{array}{ll} -(k+1) & \hbox{ if }t \leq -(k+2),\\
\\
t &\hbox{ if } -k \leq t \leq k,\\
\\
k+1 &\hbox{ if }t \geq k+2,
\end{array}
\right.
$$
and set
$$
\begin{array}{ll}
\displaystyle{\hat{v}_{k,p}= u_1 + \chi_k(v_p-u_1),}\\
\displaystyle{\hat{v}_k = u_1+ \chi_k (u-u_1).}
\end{array}
$$
Let $\varphi \in C^\infty_0(B)$ with $\varphi = 1$ in $A$ and define, for $t \in ]0,1[$,
$$
\begin{array}{ll}
\displaystyle{w_{t,k,p}=t^2(2-t)[\varphi \hat{v}_{k,p}+ (1-\varphi)u_1]+ (1-t)(1+t-t^2)u_1},\\
\\
\displaystyle{w_{t,k}=t^2(2-t)[\varphi \hat{v}_k+ (1-\varphi)u_1]+ (1-t)(1+t-t^2)u_1,}\\
\\
\displaystyle{w_t=t^2(2-t)[\varphi u+ (1-\varphi)u_1]+ (1-t)(1+t-t^2)u_1}
\end{array}
$$
Clearly, for every $p >1$, $k \in \mathbb N$, $t \in ]0,1[, w_{t,k,p}\in W^{1,p}_{\rm loc}(\mathbb R^N)$ and $w_{t,k,p}= u_1$ a.e. in $O \setminus B$.
By the convexity of $W$, the convexity and increasing monotonicity of $s \in \mathbb R^+ \to s^p \in \mathbb R^+$, we get
\begin{equation}\label{23DAT}
\begin{array}{ll}
\displaystyle{\int_O W^p(\nabla w_{t,k,p})dx \leq}\\
\\
\displaystyle{t \int_O W^p(t(2-t)(\varphi \nabla \hat{v}_{k,p}+ (1-\varphi )\nabla u_1+ ( \hat{v}_{k,p}- u_1) \nabla \varphi)  dx+}
\\
\\
\displaystyle{(1-t)\int_O W^p((1-t+t^2)\nabla u_1)dx \leq}\\
\\
\displaystyle{t^2(2-t)\int_O W^p(\varphi \nabla \hat{v}_{h,k}+ (1-\varphi) \nabla u_1)dx+}\\
\\
\displaystyle{t(1-t(2-t)) \int_O W^p\left(\frac{t(2-t)}{1-t(2-t)} (\hat{v}_{k,p}- u_1)\nabla \varphi\right)+}\\
\\
\displaystyle{(1-t)\int_O W^p((1+t-t^2) \nabla u_1) dx,}
\end{array}
\end{equation}
for every $p >1$, $ k \in \mathbb N, t \in ]0,1[.$

\noindent The estimate of the first term in the right hand side of \eqref{23DAT}, gives, since $W^p(\cdot)$ is convex,
\begin{equation}\label{24DAT}
\begin{array}{ll}
\displaystyle{\int_O W^p(\varphi \hat{v}_{k,p}+(1-\varphi) \nabla u_1) dx=}\\
\\
\displaystyle{\int_A W^p(\nabla \hat{v}_{k,p})dx + \int_{O \cap (B \setminus A) }W^p(\varphi \nabla \hat{v}_{k,p})+(1-\varphi) \nabla u_1) dx + \int_{O \setminus B}W^p(\nabla u_1) dx \leq}\\
\\
\displaystyle{\int_A W^p(\nabla \hat{v}_{k,p})dx + \int_{O\cap (B\setminus A )} \varphi W^p(\nabla \hat{v}_{k,p}) dx+}\\
\\
\displaystyle{\int_{O \cap (B \setminus A)} (1-\varphi)W^p(\nabla u_1)dx + \int_{O \setminus B} W^p(\nabla u_1) dx\leq }
\\
\\
\displaystyle{\int_{O \cap B} W^p(\nabla \hat{v}_{k,p}) dx+ \int_{O \setminus A}W^p(\nabla u_1)dx =}\\
\\
\displaystyle{\int_{O \cap B \cap \{ |v_p - u_1| \geq k+2\}} W^p(\nabla u_1) dx + \int_{O \cap B \cap \{|v_p -u_1|\leq k\}} W^p(\nabla v_p )dx +}\\
\\
\displaystyle{\int_{O \cap B \cap \{ k < |v_p -u_1|< k+2\}} W^p(\chi'_k (v_p- u_1) \nabla v_p+ (1-\chi'_k(v_p -u_1)) \nabla u_1) dx+}\\
\\
\displaystyle{\int_{O \setminus A} W^p(\nabla u_1) dx \leq}\\
\\
\displaystyle{\int_{O \cap B \cap \{ |v_p - u_1| \geq k+2\}} W^p(\nabla u_1) dx + \int_{O \cap B \cap \{|v_p -u_1|\leq k\}} W^p(\nabla v_p )dx +}\\
\\
\displaystyle{\int_{O \cap B \cap \{ k < |v_p -u_1|< k+2\}}\left[ \chi'_k (v_p-u_1)W^p(\nabla v_p) + (1-\chi'_k(v_p -u_1))W^p(\nabla u_1)\right]dx +}\\
\\
\displaystyle{\int_{O \setminus A} W^p(\nabla u_1) dx \leq}\\
\\
\displaystyle{\int_{O \cap B \cap \{|v_p -u_1|< k+2\}} W^p(\nabla v_p) dx + \int_{O \cap B \cap \{ |v_p -u_1|> k\}}W^p(\nabla u_1)dx + \int_ {O \setminus A} W^p(\nabla u_1) dx},
\end{array}
\end{equation}  
for every $p >1$, $k \in \mathbb N$, $t \in ]0,1[$.
The growth condition on $W$, expressed in \eqref{Wgrowth}, the fact that $u_1  \in W^{1,p}_{\rm loc}(\mathbb R^N)$ entail that
\begin{equation}\label{26DAT}
\displaystyle{\limsup_{ p \to 1} \int_{O \cap B \cap \{ |v_p- u_1|> k\}} W^p(\nabla u_1) dx \leq \int_{O \cap B \cap \{ |u-u_1| \geq k\}} W(\nabla u_1)dx.}
\end{equation}


Again Lebesgue's dominated convergence theorem implies that
\begin{equation}\label{26bisDAT}
\displaystyle{\lim_{p \to 1}\int_{O\setminus A}W^p(\nabla u_1) dx=\int_{O\setminus A}W(\nabla u_1)dx.}
\end{equation}

Consequently by \eqref{24DAT}, \eqref{22DAT}, \eqref{26DAT}, \eqref{26bisDAT}, the fact that $u$  coincides with the Sobolev function $u_1$ a.e. in $O \setminus A$, we obtain
\begin{equation}\label{27DAT}
\begin{array}{ll}
\displaystyle{\limsup_{p \to 1} \int_O W^p(\varphi \nabla \hat{v}_{k,p}+ (1-\varphi)\nabla u_1)dx \leq}\\
\\
\displaystyle{\leq \int_{O \cap B}W\left(\frac{d Du}{d |D u|}\right)d |D u|+}\\
\\
\displaystyle{+\int_{O\cap B \cap \{|u-u_1|\geq k\}} W(\nabla u_1)dx + \int_{O\setminus A}W(\nabla u_1)dx=}
\\
\\
\displaystyle{=\int_{O\cap B} W(\nabla u)dx + \int_A W\left(\frac{d D^s u}{d |D^s u|}\right)d |D^s u| + \int_{O \cap \partial A} W((u-u_1){\bf n})d {\cal H}^{N-1}+}\\
\\
\displaystyle{\int_{O\cap B \cap \{|u-u_1|\geq k\}} W(\nabla u_1) dx + \int_{O\setminus A}W(\nabla u_1)dx,}
\end{array}
\end{equation}
\noindent for every $k \in \mathbb N, t \in ]0,1[$.

Let us fix $k \in \mathbb N$, $t \in ]0,1[$ and observe that $\|\hat{v}_{k,p}-u_1\|_{L^\infty(O\cap B)}\leq k+2$  for every $p >1$. Therefore, the growth condition on $W$ \eqref{Wgrowth}, its convexity and the fact that $\frac{t(2-t)}{1-t(2-t)}(\hat{v}_{k,p}- u_1)\nabla \varphi \in L^\infty(O\cap B)$ converges pointwise a.e. in $O \cap B$ towards $\frac{t(2-t)}{1-t(2-t)}(\hat{v}_k- u_1)\nabla \varphi $, lead us, via Lebesgue's dominated convergence theorem, to get 
\begin{equation}\label{28DAT}
\displaystyle{\lim_{p \to 1} \int_{O\cap B}W^p\left(\frac{t(2-t)}{1-t(2-t)}(\hat{v}_{k,p}- u_1)\nabla \varphi \right)dx= \int_{O\cap B}W\left(\frac{t(2-t)}{1-t(2-t)}(\hat{v}_k- u_1)\nabla \varphi \right)dx}
\end{equation}
for every $k \in \mathbb N$, $t \in ]0,1[$.

Consequently by \eqref{23DAT}, \eqref{27DAT}, \eqref{28DAT}, we obtain
\begin{equation}\nonumber
\begin{array}{ll}
\displaystyle{\limsup_{p\to 1}\int_O W^p(\nabla w_{t,k,p})dx \leq}\\
\\
\displaystyle{t^2(2-t)\left[\int_{O \cap B} W(\nabla u)dx + \int_A W\left(\frac{d D^s u}{d |D^s |u}\right)d |D^s u|+ \right.}\\
\\
\displaystyle{\int_{O\cap \partial A} W((u_1-u){\bf n})d {\cal H}^{N-1}+}\\
\\
\displaystyle{\left. \int_{O \cap B \cap \{|u-u_1|\geq k\}} W(\nabla u_1)dx + \int_{O \setminus A} W(\nabla u_1)dx\right]+}\\
\\
\displaystyle{t(1-t(2-t))\int_O W\left(\frac{t(2-t)}{1-t(2-t)}(\hat{v}_k- u_1)\nabla \varphi \right)dx + (1-t)\int_O W\left((1+t-t^2)\nabla u_1\right) dx }
\end{array}
\end{equation}
for every $k \in \mathbb N, t \in ]0,1[.$

The proof from now on is identical to that of Lemma 2.1 in \cite{DAT}  and we omit the details. We just observe that the positive $1$-homogeneity of $W$ allows us to replace the recession function $W^\infty$ in \cite{DAT} by $W$. 
 
Thus we have that
\begin{equation}\label{213DAT}
\begin{array}{ll}
\displaystyle{\limsup_{t \to 1}\limsup_{k \to + \infty}\limsup_{p \to 1}\int_O W^p(\nabla w_{t,k,p}) dx\leq}\\
\\
\displaystyle{\int_{O \cap B}W(\nabla u)dx + \int_A W\left(\frac{d D^s u}{d |D^s u|}\right)d |D^s u|+}\\
\\
\displaystyle{\int_{O \cap \partial A}W((u_1-u){\bf n})d {\cal H}^{N-1}+ \int_{O \setminus A} W(\nabla u_1)dx. }
\end{array}
\end{equation} 
By \eqref{213DAT} the thesis follows by a standard diagonal argument once observed that
$w_{t,k,p}\to w_{t,k}$ in $L^{1^\ast}(O)$ for every $k \in \mathbb N$, $t \in ]0,1[$ as $p \to 1$, and $w_{t,k} \to w_t$ in $L^{1^\ast}(O)$ for every $t \in ]0,1[$ as $k \to +\infty$ and $w_t \to u$ in $L^{1^\ast}(O)$ as $t \to 1$.

\end{proof}

The result proved below is analogous to \cite[Lemma 2.2]{DAT}.
\begin{Lemma}\label{comelemma22DAT}
Let ${\bar p}>1$, let $O$ be a cone-shaped open set pieceiwise $C^1$, with vertex $x_0$ and basis $S$, $W:\mathbb R^N \to [0,+\infty[$ be convex, positively $1$-homogeneous and verifying \eqref{Wgrowth}, $u_1 \in W^{1,{\bar p}}_{loc}(O)$, and $u \in BV(O)$, then there exists $\{u_p\}_p$, ${\bar p}\geq p >1$, such that $u_p \in W^{1,p}_{\rm loc}(\mathbb R^N)$ and $u_p \equiv u_1$ a.e. in $C_{x_0,S}\setminus O$ for every $p>1$, $u_p \to u$ in $L^{1^\ast}(O)$ and
\begin{equation}\nonumber
\begin{array}{ll}
\displaystyle{\limsup_{p \to 1}\int_O W^p(\nabla u_p)dx \leq}\\
\\
\;\;\,\;\;\;\;\;\;\,\;\;\displaystyle{\int_O W(\nabla u)dx + \int_O W\left(\frac{d Du}{d |D u|}\right)d |D u|+ \int_S W((u-u_1){\bf n})d {\cal H}^{N-1}.}
\end{array}
\end{equation}
\end{Lemma}
\begin{proof}[Proof]

The proof is very similar to that of Lemma 2.2 in \cite{DAT}. We do not propose it in its entirety but we just outline the main steps and differences. 

First we extend $u \in C_{x_0,S}$ by defining $u = u_1$ a.e. in $C_{x_0, S}\setminus O$. Let $t \in ]1,+\infty[$ and $\tau \in ]0,1[$, set $O_t=x_0+ \frac{(O- x_0)}{t}$ and define $u_{t, \tau}= u_1 + \frac{\tau}{t}(u-u_1)(x_0+ t (\cdot- x_0))$. We have that if $x \not \in O_t$, then $x_0+ t(x-x_0) \not \in O$, hence being $u=u_1$ a.e. in $C_{x_0,S} \setminus O$, it turns out that $u_{t, \tau} \in BV(A)$ for every bounded subset $A$ of $C_{x_0,S}$, $u_{t,\tau}=u_1$ a.e. in $C_{x_0,S}\setminus O_t$ and 
$$
\displaystyle{\nabla u_{t,\tau}(x)= \nabla u_1(x)+ \tau \nabla (u-u_1)(x_0+ t(x-x_0)) \hbox{ for a.e. }x \in C_{x_0, S}.} 
$$ 
Then, instead of invoking \cite[Lemma 2.1]{DAT}, we refer to our Lemma \ref{comelemma21DAT}, which guarantees convergence in $L^{1^\ast}(O)$ of $u_{t,\tau}$ to $u$ first as $t \to 1$ and then $\tau \to 1$.
  
Indeed, by Lemma \ref{comelemma21DAT}, applied to $A= O_t$ and $B$ a bounded open set satisfying $O_t \subset \subset B$, $O \setminus {\overline B} \not = \emptyset$ and such that $O \cap B$ is piecewise $C^1$, there exists $\{u_p^{t,\tau}\}_{p, t,\tau}$ such that $u_p^{t,\tau} \in W^{1,p}_{\rm loc}(\mathbb R^N)$ and verifies $u_p^{t,\tau} \to u_{t,\tau}$ in $L^{1^\ast}(O)$, $u_p^{t,\tau}= u_1$ a.e. in $O \setminus B$ for every $p >1$, and 
\begin{equation}\label{214DAT}
\begin{array}{ll}
\displaystyle{\limsup_{p \to 1}\int_O W^p(\nabla u_p^{t, \tau})dx \leq}
\\
\\
\displaystyle{\leq \int_{O \cap B}W(\nabla u_{t,\tau}) dx + \int_{O_t}W\left(\frac{d D^s u_{t,\tau}}{d |D^s u_{t,\tau}|}\right)d |D^s u_{t,\tau}|+}
\\
\\
\displaystyle{\int_{O \cap \partial O_t} W((u_1- u_{t,\tau}){\bf n})d{\cal H}^{N-1}+ \int_{O \setminus O_t}W(\nabla u_1)dx.}
\end{array}
\end{equation}
Observe also that it is not restrictive to assume that $u_p^{t,\tau}= u_1$  a.e. in $C_{x_0,S} \setminus O$ for every $p >1$.

Then, exploiting the convexity and the positive $1$-homogeneity of $W$ and the change of variable $y= x_0 + t(x-x_0)$ the proof develops along the same lines of \cite[Lemma 2.1]{DAT}, thus we omit it.

In conclusion, 
taking first the limit as $t \to 1$ and then letting $\tau$ go to $1$, we have, 
\begin{equation}\label{222DAT}
\begin{array}{ll}
\displaystyle{\limsup_{\tau \to 1}\limsup_{t \to 1}\left\{\int_{O \cap B} W(\nabla u_{t, \tau})dx + \int_{O_t}W\left(\frac{d D^s u_{t,\tau}}{d |D^s u_{t,\tau}|}\right)d|D^s u_{t,\tau}|+\right.}\\
\\
\displaystyle{\left. \int_{O \cap \partial O_t}W((u_1-u){\bf n}) d {\cal H}^{N-1} + \int_{O \setminus O_t}W(\nabla u_1)dx\right\}\leq}\\
\\
\displaystyle{\leq \int_O W(\nabla u)dx + \int_O W\left(\frac{d D^s u}{d |D^s u|}\right)d |D^s u|+ \int_S W((u_1-u){\bf n})d {\cal H}^{N-1}.}
\end{array}
\end{equation}
By \eqref{214DAT}, \eqref{222DAT} and a diagonal argument, the thesis follows.
\end{proof}

The following result, developped in the same spirit of \cite[Lemma 3.2]{MZ} will be exploited in the sequel. 
\begin{Lemma}\label{lemma32MZ}
Let $O$ be a cone-shaped open set in $\mathbb R^N$ pieceiwise $C^1$ with vertex $x_0$ and basis $S$, satisfying \eqref{212MZ}. Let $I$ be another bounded open set in $\mathbb R^N$ piecewise $C^1$ such that $\Gamma:= S \cap I \not= \emptyset$, and assume that ${\cal H}^{N-1}({\bar \Gamma}\setminus \Gamma)= 0$. Let $W:\mathbb R^N \to [0. +\infty[$ be a convex, positively $1$-homogeneous function satisfying \eqref{Wgrowth}. Let $u \in BV(O)$ and $u_1 \in W^{1,{\bar p}}_{\rm loc}(\mathbb R^N)$ for some ${\bar p}>1$. Then there exists a sequence  $\{u_p\}_p$ such that for every $1<p<{\bar p}$, $u_p \in W^{1,p}_{\rm loc}(\mathbb R^N)$ and $u_p =u_1$ on $C_{x_0, \Gamma}\setminus O$, $u_p \to u_1$ in $L^{1^\ast}(O)$ and
\begin{equation}\nonumber
\displaystyle{\limsup_{p \to 1}\int_O W^p(\nabla u_p)dx \leq \int_O W\left(\frac{d D u}{d |Du|}\right)d |D u|+ \int_{\Gamma}W((u_1-u){\bf n})d {\cal H}^{N-1}.}
\end{equation}
\end{Lemma}
\begin{proof}[Proof.] Without loss of generality we may assume that the right hand side of \eqref{lemma32MZ} is finite.

Let $\{B^\e\}_{\e >0}$ be a decreasing family of open subsets of $O$ with Lipschitz boundary such that, setting $\Gamma^\e= B^\e \cap S$, one has,
\begin{itemize}
\item[i)] $\Gamma^\e \supset {\bar \Gamma}$;
\item[ii)] $\cap_{\e >0} \Gamma^\e = {\bar \Gamma}$.
\end{itemize}
Let $A$ be a cone-shaped set of basis $\Gamma$ and vertex $x_A$ with $x_A  \in {\rm int}(O)$, and for every $\e$ denote by $A^\e$ a cone-shaped set of basis $\Gamma^\e$ and vertex $x_\e$ , with $x_\e \in {\rm int}(O\setminus A)$. Assume that $\{A^\e\}_{\e>0}$ is a decreasing family of sets such that $\cap_{\e >0}A^\e= {\bar A}$. \eqref{Wgrowth} allows us to apply Lemma \ref{comelemma22DAT}. Hence there exists a sequence $\{u_p^\e\}_p$ with $u_p^\e \in W^{1,p}_{\rm loc}(\mathbb R^N)$, $u_p^\e = u_1$ in $C_{x_\e, \Gamma^\e}\setminus A^\e$, such that $u_p^\e \to u$ in $L^{1^\ast}(A^\e)$ and 
\begin{equation}\label{34MZ}
\displaystyle{\limsup_{p \to 1}\int_{A^\e}W^p(\nabla u_p^\e)dx \leq \int_{A^\e} W\left(\frac{d D u}{d |D u|}\right)d |D u| +\int_{\Gamma^\e}W((u_1-u){\bf n})d{\cal H}^{N-1}.}
\end{equation}

Moreover an argument analogous to that exploited in Proposition  \ref{prop3.2Aimeta} guarantees that there exists a sequence $\{v_p\}_p$ such that $v_p \in W^{1,p}_{\rm loc}(\mathbb R^N)$ and $\lim_{p \to 1}\int_{O\setminus {\bar A}} |v_p-u|^{1^\ast} dx =0$ and 
\begin{equation}\label{35MZ}
\displaystyle{\lim_{p \to 1}\int_{O\setminus {\bar A}} W^p(\nabla u_p)dx = \int_{O\setminus {\bar A}} W\left(\frac{d Du}{d |D u|}\right)d |Du|.}
\end{equation}
For every $\e>0$, let $0\leq \varphi^\e\leq 1$ be a smooth function such that 
$$
\varphi^\e: x \in O \to \left\{\begin{array}{ll} 0 &\hbox{ if }x \in C_{x_A, \Gamma},\\
1 &\hbox{ if }x \in C_{x_0,S}\setminus C_{x_0, \Gamma^\e},
\end{array}
\right.
$$
and set $w_p^\e=(1-\varphi^\e)u_p^\e + \varphi^\e v_p.$

Let us fix $\e>0$ and observe that
\begin{equation}\label{36MZ}
\begin{array}{ll}
\displaystyle{\limsup_{p\to 1} \int_O W^p(\nabla w_p^\e)dx = \limsup_{p\to 1} \left\{\int_A W^p (\nabla u_p^\e)dx +\right.}\\
\\
\displaystyle{\left. + \int_{A^\e \setminus A} W^p ([(1-\varphi^\e)\nabla u_p^\e+ \varphi^\e \nabla v_p + \nabla \varphi^\e (v_p-u_p^\e)])dx + \int_{O\setminus A^\e}W^p(\nabla v_p)dx\right\}.}
\end{array}
\end{equation}
Next, by exploiting the convexity of $W$ and \eqref{Wgrowth}, we obtain the following inequality due to the local Lipschitz continuity of $W^p$ (the constant $C$ below may vary from line to line, being uniformly bounded in $p$ as $p$ tends to $1$).
$$
\begin{array}{ll}
\displaystyle{ \int_{A^\e \setminus {\bar A}} W^p ([(1-\varphi^\e)\nabla u_p^\e+ \varphi^\e \nabla v_p + \nabla \varphi^\e (v_p-u_p^\e)])dx\leq}\\
\\
\displaystyle{\int_{A^\e \setminus A} W^p ((1-\varphi^\e)\nabla u_p^\e + \varphi^\e \nabla v_p)dx + }\\
\\
\displaystyle{C \int_{A^\e \setminus A}\left(\left|[(1-\varphi^\e)\nabla u_p^\e+ \varphi^\e \nabla v_p + \nabla \varphi^\e (v_p-u_p^\e)]\right|^{p-1} + \left|(1-\varphi^\e)\nabla u_p^\e + \varphi^\e \nabla v_p\right|^{p-1}\right) \left|\nabla \varphi^\e(v_p -u_p^\e) \right|dx .}
\end{array}
$$
By exploiting again the convexity of $W$ and Hoelder inequality we obtain
$$
\begin{array}{ll}
\displaystyle{\int_{A^\e \setminus {\bar A}} W^p ([(1-\varphi^\e)\nabla u_p^\e+ \varphi^\e \nabla v_p + \nabla \varphi^\e (v_p-u_p^\e)])dx\leq}\\
\\
\displaystyle{\int_{A^\e \setminus A} W^p(\nabla u_p^\e)dx + \int_{A^\e \setminus A}W^p (\nabla v_p)dx +}\\
\\
\displaystyle{C \left(\int_{A^\e \setminus A}\left(\left|[(1-\varphi^\e)\nabla u_p^\e+ \varphi^\e \nabla v_p + \nabla \varphi^\e (v_p-u_p^\e)]\right|^{p-1} + \left|(1-\varphi^\e)\nabla u_p^\e + \varphi^\e \nabla v_p\right|^{p-1}\right)^{\frac{p}{p-1}}dx\right)^{\frac{p-1}{p}} \cdot}
\\
\\
\displaystyle{\left(\int_{A^\e \setminus A} \left|\nabla \varphi^\e(v_p -u_p^\e) \right|^p dx\right)^{\frac{1}{p}}.}
\end{array}
$$
Thus, the last inequality and \eqref{36MZ} provide
$$
\begin{array}{ll}
\displaystyle{\limsup_{p \to 1}\int_O W^p(\nabla w_p^\e)dx \leq}\\
\\
\displaystyle{\limsup_{p\to 1}\left[\int_{A^\e}W^p(\nabla u_p^\e)dx + \int_{O \setminus A}W^p (v_p)dx + \right.}\\
\\
\displaystyle{C \left(\int_{A^\e \setminus A}\left(\left|[(1-\varphi^\e)\nabla u_p^\e+ \varphi^\e \nabla v_p + \nabla \varphi^\e (v_p-u_p^\e)]\right|^{p-1} + \left|(1-\varphi^\e)\nabla u_p^\e + \varphi^\e \nabla v_p\right|^{p-1}\right)^{\frac{p}{p-1}}dx\right)^{\frac{p-1}{p}} \cdot}\\
\\
\displaystyle{\left. \left(\int_{A^\e \setminus A} \left|\nabla \varphi^\e(v_p -u_p^\e) \right|^p dx\right)^{\frac{1}{p}}\right].}
\end{array}
$$
Exploiting \eqref{34MZ},\eqref{35MZ}, the bounds on $\int_{A^\e\setminus A}|\nabla u^\e_p|^p dx $ and $\int_{A^\e \setminus A}|\nabla v_p|^p dx$ following from \eqref{34MZ} and \eqref{35MZ} and the growth from below of $W$ in \eqref{Wgrowth}, and since both $u_p^\e$ and $v_p \in W^{1,p}_{\rm loc}(\mathbb R^N)$ converge to $u$ in $L^{1^\ast}(A^\e)$ and $L^{1^\ast}(O\setminus A)$ respectively, we can conclude, passing to the limit  as $p \to 1$, that 
$$
\displaystyle{\lim_{p \to 1}\int_O \left|w_p^\e- u\right|^{1^\ast}dx = 0,}
$$ 
\noindent obtaining also
\begin{equation}\nonumber
\begin{array}{ll}
\displaystyle{\limsup_{p\to 1}\int_O W^p (\nabla w_p^\e)dx \leq}\\
\\
\displaystyle{\int_O W\left(\frac{d D u}{d |Du|}\right)d|Du|+ \int_{\Gamma^\e}W((u_1-u){\bf n})d{\cal H}^{N-1}+}\\
\\
\displaystyle{\int_{A^\e\setminus A}W(\nabla u)dx + \int_{A^\e \setminus {\bar A}} W\left(\frac{d D^s u}{d|D^s u|}\right)d |D^s u|.}
\end{array}
\end{equation}
We also observe that
\begin{equation}\nonumber
\begin{array}{ll}
\displaystyle{\int_{\Gamma^\e} W((u_1-u){\bf n})d {\cal H}^{N-1}= \int_{\Gamma} W((u_1-u){\bf n})d {\cal H}^{N-1}+}\\
\\
\displaystyle{\int_{\Gamma^\e \setminus \Gamma}W((u_1-u){\bf n})d {\cal H}^{N-1},}
\end{array}
\end{equation}
where this latter term is finite as a consequence  of \eqref{Wgrowth}. Then the thesis follows exploiting again the growth condition and the fact that ${\cal H}^{N-1}({\bar \Gamma}\setminus \Gamma)=0$ as $A^\e$ shrinks to $A$.
\end{proof}

The following result, analogous to \cite[Lemma 3.3]{MZ}, allows us to obtain the upper bound inequality for the desired $\Gamma$-convergence.

\begin{Lemma}\label{lemma33MZ}
Let $\omega $ be a bounded open set piecewise $C^1$ and let $\Omega := \omega \times \left(-\frac{1}{2}, \frac{1}{2}\right)$. Let $W:\mathbb R^3 \to [0,+\infty[$ be a convex, positively $1$-homogeneous function verifying \eqref{Wgrowth}. Let $u \in BV(\omega)$ and $u_0 \in X(\partial \omega)= W^{\frac{{\bar p }-1}{\bar p}, {\bar p}}(\partial \omega)$ for some $1<{\bar p}$. Then there exists a sequence $\{u_p\}_p$ such that $u_p \in W^{1,p}_{u_0, {\rm lat}}\left(\omega \times \left(-\frac{1}{2},\frac{1}{2}\right)\right)$, with $1<p <{\bar p}$ such that 
\begin{equation}\nonumber
\displaystyle{\lim_{p \to 1} \int_\Omega |u_p- u|^{1^\ast}dx = 0},
\end{equation}
and 
\begin{equation}\nonumber
\displaystyle{\lim_{p \to 1} \int_\Omega W^p(\nabla u_p)dx \leq \int_\O W\left(\frac{d D u}{d |D u|}\right)d |D u| + \int_{\partial \omega \times \left(-\frac{1}{2},\frac{1}{2}\right)} W((u_0 - u)\nu)d {\cal H}^2.}
\end{equation}
\end{Lemma}
\begin{proof}[Proof.]
The proof develops along the lines of \cite[Lemma 3.3]{MZ} and we refer to it for the details. We point out only the main differences: here we exploit the positive $1$-homogeneity of $W$ and the local Lipschitz continuity of $W^p$.
  
We also emphasize that the arguments essentially rely on the application of a partition of unity to glue the recovery sequences (in $L^{1^\ast}(\Omega)$ and not just in $L^1(\Omega)$ as in \cite{MZ})  for the Neumann and Dirichlet parts of $\Omega$, i. e. `lateral boundary' and  `bases' of the domain.

\end{proof}

\begin{Theorem}\label{corollary33AGZAimeta}
Let $\{F_{p,\e}\}_p$ be the functionals introduced in \eqref{Fpe}, then $\{F_{p,\e}\}_p$ $\Gamma$-converges as $p \to 1$, with respect to the $L^1(\Omega)$ strong topology to $F_{1,\e}$.
\end{Theorem}
\begin{proof}[Proof]
We start observing that Lemma \ref{lemma33MZ} guarantees the existence of a sequence $\{u_p\}_p$ such that $u_p \in W^{1,p}_{u_0,{\rm lat}}(\Omega)$,
$$
\displaystyle{\lim_{p \to 1}\int_\Omega |u_p- u|^{1^\ast}dx =0,}
$$
and
$$
\displaystyle{\lim_{p \to 1}\int_{\Omega}W_\e^p(\nabla u_p)dx = \int_\Omega W_\e\left(\frac{d D u}{d |Du|}\right)d |D u|+ \int_{\partial \omega \times \left(-\frac{1}{2}, \frac{1}{2}\right)}W_\e((u-u_0){\nu})d{\cal H}^2=}
$$
$$
\displaystyle{\left|D_\alpha u ,\frac{1}{\e}D_3 u\right|(\Omega)+ \int_{\partial \omega \times \left(-\frac{1}{2}, \frac{1}{2}\right)}W_\e((u-u_0){\nu})d{\cal H}^2.}
$$
These prove the upper bound.
For what concerns the lower bound it is enough to invoke Proposition \ref{prop31MZ}. This concludes the proof.
\end{proof}

\begin{Remark}\label{Rem34AGZAimeta}
We recall that the $\Gamma$-convergence result as $\e \to 0$, stated in Proposition \ref{gammae} is the same either if we consider the family of functionals $\{G_{1,\e}\}_{\e}$ in \eqref{G1e} or their relaxed ones $\{I_{1,\e}\}_{\e}$ in \eqref{F1e} (cf. \cite[Proposition 6.11]{DM}).
\end{Remark}

\begin{Remark}\label{rem35AGZAimeta}
Let $O\subset \mathbb R^N$ be any bounded open set with piecewise $C^1$ boundary $1< {\bar p}$, and 
let $u_1 \in W^{1-\frac{1}{p},p}(\partial O)$.
The results expressed by Proposition \ref{prop2.3Aimeta} and the arguments in the first part of that proof, allow us to prove $\Gamma$-convergence, as $p \to 1$, with respect to $L^1$-strong convergence of the functionals $\{G_p\}_p :u \in W^{1,p}_{u_1}(O) \to \int_O W^p(\nabla u)dx$ towards $G_1: u \in BV(O) \to \int_O W\left(\frac{d Du}{d |D u|}\right) d |Du| + \int_{\partial O}W(|u-u_1|{\bf n})d {\cal H}^{N-1}$ (${\bf n}$ being the unit exterior normal to $\partial O$) for any $W:\mathbb R^N \to [0,+\infty)$ convex, positively $1$-homogeneous, satisfying a linear growth condition as \eqref{Wgrowth}. 

\medskip
We also observe that in Theorem \ref{corollary33AGZAimeta} and in the preliminary lemmata, we have chosen a function $W$ positively $1$-homogeneous, having in mind the applications to the $-\Delta_1$ type equations, but the $\Gamma$ convergence results hold similarly without this assumption, introducing the recession function $W^\infty$ in the integrals dealing with the singular part of $Du$.

\end{Remark}

\begin{Remark}\label{diagram} Let $\omega \subset \mathbb R^2$ is a bounded open set, piecewise $C^1$, and let $u_0\in W^{1-\frac{1}{\bar p}, {\bar p}}(\partial \omega)$, for some ${\bar p} >1$, recall the families of problems $\{{\cal P}_{p,\e}\}_{p,\e}$, $\{{\cal P}_{1,\e}\}_\e$, $\{{\cal P}_{p,0}\}_p$ and ${\cal P}_{1,0}$  in \eqref{p-e-problem}, \eqref{1-e-problem}, \eqref{p-0-problem} and \eqref{1-0-problemBv}, respectively.

 As a consequence  of the above results we obtain that the dimensional reduction and the so-called power law approximation, namely the convergence as $p\to 1$, commute in the sense of $\Gamma$- convergence with respect to $L^1(\Omega)$-strong convergence, as  summarized by the following diagram:

\begin{picture}(400,160)(-200,-100)
\put(-20,30){\vector(1,0){80}}
\put(-50,28){${\cal P}_{p,\e}$}
\put(70,28){${\cal P}_{1,\e}$}
\put(0,16){$p\longrightarrow 1$}
\put(-40,20){\vector(0,-1){80}}
\put(-30,-10)
{$\e\longrightarrow 0$}
\put(75,20){\vector(0,-1){80}}
\put(80,-10)
{$\e\longrightarrow 0$}
\put(-50,-78){${\cal P}_{p,0}$}
\put(-20,-75){\vector(1,0){80}}
\put(0,-89){$p\longrightarrow 1$}
\put(70,-78){${\cal P}_{1,0}$}
\end{picture}

Indeed, the  vertical arrows have been treated in Proposition \ref{gammae}, the upper horizontal arrow has been  proved in Theorem \ref{corollary33AGZAimeta} while the lower horizontal arrow follows from Theorem \ref{Gammap0to10}.

Other types of commutativity of solutions to problems ${\cal P}_{p,\e}$ as $p\to 1$ and $\e \to 0$ will be discussed in the following sections.
\end{Remark}

\begin{Remark}\label{Gammaconvergenceinmonotonicity}
Let $\omega \subset \mathbb R^2$ be a bounded open set, piecewise $C^1$, with ${\cal L}^2(\omega)=1$ for convenience, let $\Omega:= \omega \times \left(-\frac{1}{2},\frac{1}{2}\right)$, let $W: \mathbb R^3 \to \mathbb [0,+\infty[$ be a continuous function, positively $1$-homogeneous and verifying \eqref{Wgrowth}. Fix $\overline{p}>1$,  and let $u_0 \in W^{1,\overline{p}}_{\rm loc}(\mathbb R^2)$. For every $1< p\leq \overline{p}$ we can define the functionals
\begin{equation}\nonumber
{\cal F}_p(u)=\left\{ 
\begin{array}{ll} 
\left\| W(\nabla u)\right\|_{L^p(\Omega)} & \hbox{ if }u \in W^{1,p}_{u_0, {\rm lat}}(\Omega),\\
+\infty & \hbox{otherwise in }BV(\Omega).
\end{array}
\right.
\end{equation}
It is easily verified that ${\rm dom}({\cal F}_p) \supset {\rm dom}({\cal F}_{q}) $ whenever $1< p <q$  and if $u \in {\rm dom}({\cal F}_q)$ then ${\cal F}_p(u) \leq {\cal F}_q(u)$.

Let ${\cal F}:BV(\Omega) \to [0,+\infty]$ be the functional defined as
\begin{equation}\nonumber
{\cal F}(u)=\left\{ 
\begin{array}{ll}
\left\|W(\nabla u)\right\|_{L^1(\Omega)} & \hbox{ if } u \in \bigcup_{p > 1}W^{1,p}_{u_0, {\rm lat}}(\Omega), \\
\\
+ \infty &\hbox{ if }u \in BV(\Omega)\setminus \bigcup_{p > 1}W^{1,p}_{u_0, {\rm lat}}(\Omega).
\end{array}
\right.
\end{equation}
The monotonicity of $\{{\cal F}_p\}_p$ provides pointwise convergence as $p \to 1$ of ${\cal F}_p(u)$ towards ${\cal F}(u)$, for every  $u \in BV(\Omega)$. On the other hand it is easy to verify that ${\cal F}$ is not lower semicontinuous with respect to $L^1(\Omega)$ strong convergence. Thus \cite[Proposition 5.7]{DM} ensures $\Gamma$ convergence, with respect to $L^1(\Omega)$ strong convergence, of ${\cal F}_p$, as $p \to 1$, towards the lower semicontinuous envelope of ${\cal F}$, denoted by $\overline{\cal F}$. On the other hand  Lemma \ref{prop31MZ} and Lemma \ref{lemma33MZ} guarantee that  $\{{\cal F}_p\}_p$ $\Gamma$ converges, with respect to  $L^1(\Omega)$ strong convergence, as $p \to 1$, to the functional
\begin{equation}\nonumber
\displaystyle{{\cal F}_1(u)=\int_\Omega W(\nabla u)dx + \int_\Omega W\left(\frac{d D^s u}{d |D^s u|}\right)d |D^s u|+\int_{\partial \omega \times \left(-\frac{1}{2},\frac{1}{2}\right)} W((u-u_0) \nu) d {\cal H}^2}
\end{equation}
for every $u \in BV(\Omega)$. 
Consequently we have proven that ${\overline{\cal F}}(u)= {\cal F}_1(u)$ for every $u \in BV(\Omega)$.
\end{Remark}

\section{Asymptotics in terms of differential problems}\label{dualityapproach}

Formally, putting $p = 1$ in (\ref{p-e-equation}) and (\ref{p-0-equation-2}), one obtains
\begin{equation}\label{Delta_1_e}
\left\{\begin{array}{ll}
\displaystyle{-\Delta_{1,\e}u=-{\rm div}\left(|{ Id}_{\e}\nabla u \cdot \nabla u|^{\frac{-1}{2}}{Id}_\e\nabla u\right)= 0} & \hbox{ in }\omega,\\
\\
u= u_0 & \hbox{ on }\partial \omega \times \left(-\frac{1}{2}, \frac{1}{2}\right),
\\
\\
|Id_{\e}\nabla u\cdot \nabla u|^{\frac{-1}{2}}(Id_{\e}\nabla u)\cdot \nu=0 &\hbox{ on }\omega \times\left\{-\frac{1}{2},\frac{1}{2}\right\}
\end{array}
\right.
\end{equation}
where ${ Id}_{\e}$ has been defined in \eqref{Ide}, and
\begin{equation}\label{Delta_1_0}
\left\{\begin{array}{ll}
\displaystyle{-\Delta_{1,0}u=-{\rm div}\left(\frac{\nabla u}{|\nabla u|}\right)= 0} & \hbox{ in }\omega,\\
\\
u= u_0 & \hbox{ on }\partial \omega.
\end{array}
\right.
\end{equation}

\noindent Clearly the above equations are meaningless in $W^{1,1}(\omega)$.
In order to deal  with  problems \eqref{p-e-problem} and \eqref{1-0-problemBv} in terms of PDE's it is useful to approach them via the duality theory developed by Ekeland and Temam in the context of variational problems (see \cite{ET}).

The following proposition is stated in \cite[Proposition 1.1]{KT} and, with the purpose of applications to $1$-Laplace equations quoted also in \cite{D0, D,D2}.  A proof can be found in \cite[Theorem 3.2]{KT} in the context of Hencky's Plasticity theory.

\begin{Proposition}\label{ourresult}
Let $O \subset \mathbb R^n$ be an open set. 
Suppose that $u \in BV(O)$ and $\sigma \in L^\infty(O;\mathbb R^n)$ is such that ${\rm div}\sigma \in L^n(O)$. One defines the distribution $\sigma \cdot Du$ by the following
\begin{itemize}
\item [1.] For every $\varphi \in {\cal D}(O)$
\begin{equation}\nonumber
< \sigma \cdot Du, \varphi>= -\int_O {\rm div}(\sigma)u \varphi dx - \int_O \sigma \cdot D(\varphi) u dx.
\end{equation}
Then, the distribution $\sigma \cdot Du$ hence defined is a bounded measure in $O$, absolutely continuous with respect to $|Du|$, with
\begin{equation}\label{10D2}
|\sigma \cdot Du|\leq |Du ||\sigma|_\infty.
\end{equation} 
\item[2.] Suppose that $O$ is piecewise $C^1$. The following generalized Green's formula holds for
$\varphi \in {\cal D}(\mathbb R^n)$
\begin{equation}\label{11D2}
<\sigma \cdot Du, \varphi>=- \int_O {\rm div}(\sigma) u \varphi dx - \int_O \sigma \cdot D\varphi u dx  + \int_{\partial O}\sigma \cdot {\bf \nu}u\varphi d{\cal H}^{n-1},
\end{equation}
where $\nu$ denotes the unit outer normal to $\partial O$ and ${\cal H}^{n-1}$ the $n-1$ dimensional Hausdorff measure.
\item[3.] Define
$$
\sigma \cdot D^s u = \sigma \cdot Du - \sigma \cdot \nabla u,
$$
where $\nabla u$ and $D^su$ represent the absolutely continuous (with respect to the Lebesgue measure)   and singular part of $Du$. Then $(D^s u)\cdot \sigma $ is singular and 
$$
|\sigma \cdot D^s u| \leq |D^s u||\sigma|_{L^\infty}.
$$
\end{itemize}
\end{Proposition}


By virtue of Proposition \ref{ourresult} applied to $O=\omega$, one may consider the following equation
\begin{equation}\label{1-0-equation}
\left\{
\begin{array}{ll}
-{\rm div}\sigma= 0, &\hbox{ in }\omega,\\
\\
\sigma \cdot D u =|Du| &\hbox{ in }\omega,\\
\\
\sigma \cdot \nu (u-u_0)= |u -u_0| &\hbox{ on }\partial \omega,
\end{array}
\right.
\end{equation}
which represents the formal expression for $1$-harmonic functions in $\omega$ with boundary datum $u_0$, namely for \eqref{Delta_1_0}.

Applying again Proposition \ref{ourresult} to $O=\omega \times \left(-\frac{1}{2},\frac{1}{2}\right)$ we can give a meaning to the anisotropic $-\Delta_{1,\e}$ operator appearing in dimension reduction, and we can also consider it as the `Euler-Lagrange equation' associated to \eqref{Delta_1_e}.

\begin{equation}\label{1-e-equation}
\left\{
\begin{array}{ll}
-{\rm div }\sigma_\e =0 &\hbox{ in }\omega \times \left(-\frac{1}{2}, \frac{1}{2}\right),\\
\\
\sigma_\e \cdot \nabla u =|Id_\e \nabla u \cdot \nabla u|^\frac{1}{2} &\hbox{ in }\omega \times \left(-\frac{1}{2}, \frac{1}{2}\right),\\
\\
\sigma_\e \cdot \nu (u-u_0)= |u-u_0| &\hbox{ on } \partial \omega \times \left(-\frac{1}{2}, \frac{1}{2}\right),\\
\\
\sigma_\e \cdot \nu= 0 &\hbox{ on } \omega \times \left\{-\frac{1}{2}, \frac{1}{2}\right\},
\end{array}
\right.
\end{equation}
where $\nu$ represents the unit outer normal vector to $\partial \omega \times \left(-\frac{1}{2}, \frac{1}{2}\right)$ and $Id_\e$ is as in \eqref{Ide}.

 Via the duality theory the solutions to (\ref{1-0-equation}) and (\ref{1-e-equation}) are in correspondence with the minimizers of ${\cal P}_{1,\e}$ in \eqref{1-e-problemBv} and ${\cal P}_{1,0}$ in \eqref{1-0-problemBv}, according to the regularity assumptions on $u_0$.

In fact we can invoke Theorem \ref{Theorem ET} and apply it to   
\eqref{1-0-equation} and \eqref{1-e-equation}. Namely, having in mind the notations of Theorem \ref{Theorem ET} in the first case we can set $X= W^{1,1}(\omega)$ and $Y=(L^1(\omega))^2 $, the linear operator $\Lambda$ maps $u \in X$ to $\nabla u \in Y$, $G$ and $F$ are defined as
$$
\displaystyle{G({\bf p})=\int_\omega \left(\sum_{i=1}^2 p_i^2\right)^{\frac{1}{2}}dx_1dx_2},
$$  
with ${\bf p}=(p_1,p_2)$
$$
\displaystyle{F(u)=\left\{
\begin{array}{ll}
0 &\hbox{ if }u \equiv u_0 \hbox{ in }\partial \omega,
\\
\\
+ \infty &\hbox{ otherwise,} 
\end{array}
\right.}
$$
where the equality is intended, as usual, in the sense of traces, recalling that $u_0 \in W^{1-\frac{1}{\bar p}, {\bar p}}(\partial \omega)$, for some ${\bar p} >1$.

Thus it easily checked that the dual Problem of ${\cal P}_{1,0}$ is 
\begin{equation}\label{dualP_1_0}
\displaystyle{{\cal D}_0=\sup_{\begin{array}{ll}
\sigma \in L^\infty(\omega;\mathbb R^3),
\\
{\rm div}\sigma =0, |\sigma|\leq 1
\end{array}}\left\{-\int_{\partial \omega}\sigma \cdot \nu u_0 d{\cal H}^1\right\},}
\end{equation}
where in fact $\sigma$ is exactly as in \eqref{1-0-equation}. 

Analogously in the $\e$-dependent case, by assuming $X$ as $W^{1,1}\left(\omega \times \left(-\frac{1}{2}, \frac{1}{2}\right)\right)$ and $Y= (L^1\left(\omega \times \left(-\frac{1}{2},\frac{1}{2}\right)\right))^3$ and $\Lambda:u \in X \to (\nabla_\alpha u, \frac{1}{\e}\nabla_3 u) \in Y$.
Let $G$ be given by
$$
\displaystyle{G({\bf p})= \int_\Omega \left(\sum_{i=1}^3p_i^2 \right)^{\frac{1}{2}}dx_1dx_2dx_3}
$$
 and 
$$
\displaystyle{F(u)=\left\{
\begin{array}{ll}
0 &\hbox{ if }u \equiv u_0 \hbox{ in }\partial \omega\times \left(-\frac{1}{2},\frac{1}{2}\right),
\\
\\
+ \infty &\hbox{ otherwise.} 
\end{array}
\right.}
$$
The dual problem becomes 
$$
\displaystyle{{\cal D}_{1,\e}=\sup_{\begin{array}{ll}
\sigma_\e \in L^\infty(\Omega;\mathbb R^3),
\\
{\rm div}\sigma_\e =0,\left |Id_{\frac{1}{\e}}\sigma_\e\right|\leq 1,\\
\sigma_\e \cdot \nu =0 \hbox{ on }\omega \times \left\{-\frac{1}{2},\frac{1}{2}\right\}
\end{array}}\left\{-\int_{\partial \omega \times \left(-\frac{1}{2},\frac{1}{2}\right)}\sigma_\e \cdot \nu u_0 d{\cal H}^1\right\}.}
$$



\begin{Remark}\label{dualityrem}
We observe that the application of Theorem \ref{Theorem ET} entails the existence of the solution only  to the dual problems, related to anysotropic almost $1$-Laplacian and almost $1$-Laplacian, namely to \eqref{1-0-equation} and \eqref{1-e-equation}. On the other hand the regularity of $u_0$, namely the fact that it is in some suitable fractional Sobolev Space, guarantees the application of our $\Gamma$-convergence results, Theorem \ref{Gammap0to10} and \ref{corollary33AGZAimeta} but not the convergence of the minimizers at $p$-level of ${\cal P}_{p,0}$ and ${\cal P}_{p,\e}$ (that exist for convexity reasons) to the infima in the original problems ${\cal P}_{1,0}$ and ${\cal P}_{1,\e}$ respectively as $p \to 1$. 
A direct proof of existence of minimizers to ${\cal P}_{1,0}$ and ${\cal P}_{1,\e}$ will be provided in the last section.
\end{Remark}

\begin{Proposition}\label{prop1.2KT}
Suppose that $u \in BV(\omega)$, and $\sigma \in L^\infty(\omega;\mathbb R^3)$, with ${\rm div }\sigma=0$ and $|\sigma|\leq 1$ a.e. in $\omega$. Then $u$ and $\sigma$ are extremal for ${\cal P}_{1,0}$ and ${\cal D}_0$, respectively if and only if 
\begin{equation}\label{1.6KT}
\displaystyle{-\sigma \cdot Du = |D u| \hbox{ as measures on }\omega,}
\end{equation}
and
\begin{equation}\label{1.7KT}
\displaystyle{\sigma \cdot \nu= \frac{u-u_0}{|u-u_0|} \hbox{ on }\partial \omega \cap \{u \not = u_0\}.}
\end{equation}
\end{Proposition}
\begin{proof}[Proof]
Using \eqref{10D2} and \eqref{11D2} with $A= {\rm Id}$ and $\varphi \in {\cal D}(\mathbb R^2)$, with $\varphi  \equiv 1$ in a neighborhood of $\omega$, and since $|\sigma \cdot \nu |\leq 1$ on $\partial \omega$, we have
\begin{equation}\label{1.8KT}
\begin{array}{ll}
\displaystyle{-\int_{\partial \omega}(\sigma \cdot \nu)u_0 d {\cal H}^1 =- \int_{\partial \omega}(\sigma \cdot \nu)u d{\cal H}^1 + \int_{\partial \omega}(\sigma \cdot \nu)(u-u_0)d {\cal H}^1=}\\
\\
\displaystyle{-\int_{\partial \omega} (\sigma \cdot \nu)(u_0-u)d {\cal H}^1 -\int_{\omega} \sigma \cdot D u \leq}\\
\\
\displaystyle{\int_{\partial \omega}|u_0-u|d {\cal H}^1 +\int_\omega |Du|.}
\end{array}
\end{equation}  
Equality holds in \eqref{1.8KT} if and only if $u$ and $\sigma$ are extremals for ${\cal P}_{1,0}$ in \eqref{1-0-problemBv} and for ${\cal D}_0$ in \eqref{dualP_1_0}, repsectively. In fact, in that case
$$
\displaystyle{\int_{\partial \omega}(\sigma \cdot \nu)(u-u_0) d {\cal H}^1 =\int_{\partial \omega} |u-u_0|d {\cal H}^1},
$$
which implies \eqref{1.7KT}, since $|\sigma \cdot \nu|\leq 1$ a.e. , and 
$$
-\int_\omega (\sigma \cdot Du)=\int_\omega |Du|,
$$
which implies \eqref{1.6KT}, since $(\sigma \cdot Du)+ |Du|$ is a nonnegative measure.
\end{proof}

A proof entirely analogous lead to the following result.
\begin{Proposition}\label{prop1.2KTRobin}
Suppose that $u \in BV\left(\omega\times \left(-\frac{1}{2},\frac{1}{2}\right)\right)$, and $\sigma \in L^\infty\left(\omega\times \left(-\frac{1}{2},\frac{1}{2}\right);\mathbb R^3\right)$, with ${\rm div }\sigma=0$ and $\left|Id_{\frac{1}{\e}}\sigma\right|\leq 1$ a.e. in $\omega\times \left(-\frac{1}{2},\frac{1}{2}\right)$. Then $u$ and $\sigma$ are extremal for ${\cal P}_{1,\e}$ and ${\cal D}_{1,\e}$, respectively if and only if 
\begin{equation}\nonumber
\displaystyle{-\sigma \cdot Du = |Id_\e Du \cdot D u|^\frac{1}{2} \hbox{ as measures on }\omega\times \left(-\frac{1}{2},\frac{1}{2}\right),}
\end{equation}
and
\begin{equation}\nonumber
\displaystyle{\sigma \cdot \nu= \frac{u-u_0}{|u-u_0|} \hbox{ on }\partial \omega\times \left(-\frac{1}{2},\frac{1}{2}\right) \cap \{u \not = u_0\},}
\end{equation}
and
\begin{equation}\nonumber
\displaystyle{\sigma \cdot \nu =0} \hbox{ on }\omega \times \left\{-\frac{1}{2},\frac{1}{2}\right\}. 
\end{equation}
\end{Proposition}

\section{Asymptotics in terms of least gradient problem}\label{LGdimred}

The target of this section consists of discussing asymptotics as $\varepsilon \to 0$ and $p \to 1$ for problems \eqref{p-e-equation} when the imposed boundary datum has a regularity, in principle different from that required in the previous $\Gamma$-convergence analysis, but a more stringent requirement is imposed on the domain $\omega \times \left(-\frac{1}{2},\frac{1}{2}\right)$.  Under this new setting of assumptions we will prove that the problems ${\cal P}_{1,\e}$ in \eqref{1-e-problemBv} and ${\cal P}_{1,0}$ in \eqref{1-0-problemBv} admit indeed a solution. Consequently in the light of Propositions \ref{prop1.2KT} and \ref{prop1.2KTRobin}, there exist solutions to the anysotropic almost $1$-Laplacian and almost $1$-Laplacian in \eqref{Delta_1_e} and \eqref{Delta_1_0}, respectively.   We recall that the symbols for the domains $\Omega$ and $\omega$ denote the same sets as in subsection \ref{RP}, namely $\omega \subset \mathbb R^2$, bounded open set and $\Omega=\omega \times \left(-\frac{1}{2},\frac{1}{2}\right).$  

As already observed in subsection \ref{RP} there is equivalence between problems \eqref{p-e-equation}  and their variational formulation \eqref{p-e-problem} when $p>1$ and the boundary datum $u_0$ is in a suitable fractional Sobolev space. This fact may be no longer true if one requires $u_0$ to be a continuous function of $\partial \omega$, cf. \cite{J}.
 
On the other hand, as already emphasized, the problems ${\cal P}_{p,\e}$ and ${\cal P}_{p,0}$ exhibit other behaviours when $p=1$, and the equivalence between the integral formulation and the differential one needs to be understood in different ways. We have already seen in section \ref{dualityapproach} the interpetration in terms of duality (see \cite{ET}).  Now we make a link in terms of least gradient functions, which will allow us to determine sufficient conditions for the existence of solutions to ${\cal P}_{1,0}$ and ${\cal P}_{1,\e}$.

\medskip
We start by focusing on the case $p>1$ and $\e=0$, and we recall the definition of $p-harmonic$ functions following \cite[Definition 2.2]{J}, namely weak solutions of \eqref{p-0-equation-2}, when $u_0 \in C(\partial \omega)$. We start by giving this definition on any generic open set $O\subset \mathbb R^n$.  

\begin{Definition}\label{def2.2J}
Let $1<p<\infty$, a continuous function $u \in W^{1,p}_{\rm loc}(O) $ is $p-harmonic$ in $O$ if
\begin{equation}\nonumber
\displaystyle{\int_O |\nabla u|^{p-2}\nabla u \cdot \nabla \varphi dx = 0}
\end{equation}
for every $\varphi \in C^\infty_0(O)$. 
\end{Definition}

\noindent The continuity in Definition \ref{def2.2J} is redundant as shown in \cite{J}.

It is useful also to recall (see \cite{J}) that a continuous function $u \in W^{1,p}_{\rm loc}(O)$ is $p-harmonic$ in $O$ if and only if 
$$
\displaystyle{\int_{O_0} |\nabla u|^p dx \leq \int_{O_0} |\nabla v|^p dx \hbox{ whenever } O_0 \hbox{ open set } \subset \subset O \hbox{ and }u-v \in W^{1,p}_0(O). }
$$ 
\noindent Now we recall some results deeply connected with $-\Delta_1$, problem \eqref{1-0-problemBv} and its approximating ones \eqref{p-e-problem}. 
The analysis we present will be mainly concerned with differential problems defined in the cross section $\omega$, when the boundary datum $u_0$ is regular.  To this end we will recall the notion of functions of least gradient in a generic open set $O\subset \mathbb R^n$.

Let $O\subset \mathbb R^n$ be an open set, following \cite{SWZ}, we say that a function $u \in BV(O)$, with prescribed boundary value $u_0 \in C(\partial O)$  is of least gradient if it is a solution of
\begin{equation}\label{leastgradient1}
\inf_{u \in BV(O)}\{|Du|(O), u \equiv u_0 \hbox{ on } \partial O\}.
\end{equation}

It has been established in \cite{SZ} that the existence of such a function is deeply related with the regularity of $O$, the regularity of the trace $u_0$ and the sense in which this trace must be understood, indeed this latter fact plays a crucial role.

In fact one may also consider 
\begin{equation}\label{leastgradient2}
\inf_{u \in BV(O) \cap C(\overline{O})}\{|Du|(O), u \equiv u_0 \hbox{ on } \partial O\}.\end{equation}

Clearly in this latter problem the trace is intended in the classical sense (restriciton), and the equality $u=u_0$ is understood pointiwise in $\partial O$. On the contrary in \eqref{leastgradient1} the equality $u=u_0$ on $\partial O$ has to be taken in the sense of traces for $BV$-functions (see subsection \ref{gammaconvergence}). 


The following result has been proven in \cite{SWZ}.

\begin{Theorem}\label{thm2.1J}

Let $O \subset \mathbb R^n$ be a bounded Lipschitz domain such that $\partial O$ has non-negative mean curvature (in a weak sense) and is not locally area-minimizing. If $u_0 \in C(\partial O)$, then there exists a unique function of least gradient $u \in BV(O)\cap C(\overline{O})$ such that $u \equiv u_0$ on $\partial O$.
\end{Theorem}

The assumptions in Theorem \ref{thm2.1J} mean that
\begin{itemize}
\item For every $x \in \partial O$ there exists $\e_0>0$ such that for every set of finite perimeter $A \subset \subset B(x,\e_0)$
\begin{equation}\label{3.1SWZ}
\displaystyle{P(O ;\mathbb R^n) \leq P(O\cup A; \mathbb R^n)}
\end{equation}
\item For every $x \in \partial O$, and every $\eta >0$ there exists a set of finite perimeter $A \subset \subset B(x, \eta)$ such that
\begin{equation}\label{3.2SWZ}
\displaystyle{P(O,B(x,\eta)) > P(O\setminus A, B(x,\eta)),}
\end{equation}
\end{itemize}
where $P(\cdot;\mathbb R^n)$ denotes the perimeter in $\mathbb R^n$.
Examples showing that neither \eqref{3.1SWZ} nor \eqref{3.2SWZ} can be dropped are given in \cite{SWZ}.

On the other hand in \cite{SZ}, (to which we refer for the precise assumptions) it has been established the following result. 

\begin{Theorem}\label{thm2.2SZ}
Let $O \subset \mathbb R^n$ be a bounded Lipschitz domain satisfying the above assumptions and a uniform exterior ball condition of radius $R$. Then there is at most one solution to the least gradient problem (\ref{leastgradient1}).
\end{Theorem}

Clearly, combinining both the assumptions in Theorems \ref{thm2.1J} and \ref{thm2.2SZ}, the solutions of problems \eqref{leastgradient1} and \eqref{leastgradient2} are unique and coincide.

In order to deal with the asymptotics as $p\to 1$ of the $-\Delta_p$- equations, Juutinen in \cite[Theorem 3.1]{J} has proven the following theorem (cf. also Remark 3.4 therein).

\begin{Theorem}\label{thm3.1J}
Let $O \subset \mathbb R^n$ be a bounded  smooth domain whose boundary has positive mean curvature and $u_0 \in C(\partial O)$, and let $u \in BV(O)\cap C(\overline{O})$ be the unique function of least gradient such that $u =u_0$ on $\partial O$. Then if $u_p \in W^{1,p}_{\rm loc}(O)\cap C(\overline{O})$ is the unique $p$-harmonic function satisfying $u_p =u_0$ on $\partial O$, it results
$$
\begin{array}{ll}
u_p \to u &\hbox{ uniformly in } O, \hbox{ as }p \to 1.
\end{array}
$$ 
\end{Theorem}

We recall that the existence and uniqueness of the solution $u_p$ mentioned in Theorem \ref{thm3.1J} relies not on `classical' Calculus of Variations arguments, since the boundary datum $u_0$ may not be the trace of a Sobolev function.  The exploited techniques are those suitably employed in the context of Nonlinear PDEs, namely the existence can be deduced as in \cite[Theorem 9.25]{HKM}, while the uniqueness derives from arguments entriely similar to the so-called \cite[{\it Comparison principle} 7.6]{HKM}.  We also stress the fact that $p \to 1$, namely it is $1<p<2$, thus a posteriori, for such regular sets $\O$, one obtain that $C(\partial O)\subset W^{1-\frac{1}{p},p}(\partial O)$.

Now we state a lemma that will be exploited in the sequel.

\begin{Lemma}\label{Marilia}
Let $O$ be a smooth domain with $\partial O = \Gamma_1 \cup \Gamma_2$ and $\Gamma_1 \cap \Gamma_2 = \emptyset$.  
Let $u,v \in W^{1,p} (O)$ satisfy
$$-\Delta_p u \ge - \Delta_p v$$

in $O$, 

$$|\nabla u |^{p-2} \frac{\partial u}{\partial \nu} \ge |\nabla v |^{p-2} \frac{\partial v}{\partial \nu}$$

on $\Gamma_1$ and

$$u \ge v$$

on $\Gamma_2$ in the strong sense. Then $u\ge v$ in $\overline{O}$
\end{Lemma}
\begin{proof}[Proof.]
In order to prove that $u\ge v$ in $\overline{O}$ it is enough to show that

$$(u-v)^{-} \equiv 0$$

To this purpose, given $\varepsilon >0$, we use $(u + \varepsilon-v)^{-}$ as a test function in the equation

$$-\left(\Delta_p u - \Delta_p v \right) \ge 0$$
and we integrate to achieve

$$-\int_{O} \left(\Delta_p u - \Delta_p v \right) \left(u+ \varepsilon-v\right)^{-}dx \le 0$$

Then we use the integration by parts to obtain

\begin{equation}\label{0.1Marilia}
\int_{O} \left(|\nabla u|^{p-2} \nabla u - |\nabla v|^{p-2} \nabla v \right) \nabla \left(u+ \varepsilon-v\right)^{-}dx - \int_{\partial O} \left(|\nabla u |^{p-2} \frac{\partial u}{\partial \nu} - |\nabla v |^{p-2} \frac{\partial v}{\partial \nu}\right) \left(u+ \varepsilon -v\right)^{-}d{\cal H}^{N-1} \le 0
\end{equation}

Recalling that

$$\int_{O}  \left(|\nabla u|^{p-2} \nabla u - |\nabla v|^{p-2} \nabla v \right) \nabla \left(u+ \varepsilon -v\right)^{-}dx $$
$$= \int_{O \cap \left\{u+ \varepsilon < v\right\}}  \left(|\nabla u|^{p-2} \nabla u - |\nabla v|^{p-2} \nabla v \right) \cdot \left(\nabla u - \nabla v\right)dx$$
$$=\int_{O \cap \left\{u+ \varepsilon < v\right\}} |\nabla u|^{p-2} \nabla u \cdot \left(\nabla u - \nabla v \right)dx - \int_{O \cap \left\{u+ \varepsilon < v\right\}} |\nabla v|^{p-2} \nabla v \cdot \left(\nabla u - \nabla v \right)dx.$$

We use the following inequality (cf. \cite[Lemma 4.2]{L})

$$|x_2 |^{p} - |x_1 |^{p} \ge p |x_1 |^{p-2}  x_1 \cdot (x_2 - x_1 ) + c(p)\frac{|x_2 -x_1 |^p}{\left(|x_1| + |x_2|\right)^{2-p}}$$

and we get

$$|\nabla u|^{p-2} \nabla u \cdot \left(\nabla u - \nabla v \right) \ge \frac{1}{p} \left[|\nabla u|^{p} - |\nabla v|^{p} +c(p) \frac{|\nabla u - \nabla v|^{p}}{\left(|\nabla u| + |\nabla v|\right)^{2-p}}\right]$$
 \noindent a.e. in $O$, 
$$-|\nabla v|^{p-2} \nabla v \cdot \left(\nabla u - \nabla v \right) \ge \frac{1}{p} \left[|\nabla v|^{p} - |\nabla u|^{p} + c(p) \frac{|\nabla u - \nabla v|^{p}}{\left(|\nabla u| + |\nabla v|\right)^{2-p}}\right]$$
\noindent a.e. in $O$.

\noindent Hence

$$\left(|\nabla u|^{p-2} \nabla u - |\nabla v|^{p-2} \nabla v \right) \cdot \left(\nabla u - \nabla v\right) \ge \frac{2}{p}c(p)\frac{|\nabla u - \nabla v|^{p}}{\left(|\nabla u| + |\nabla v|\right)^{2-p}}>0,  $$
a.e. in $O$.

\noindent This proves the positivity of the first term. For the second term, we can split it into two pieces

$$\int_{\partial O} \left(|\nabla u |^{p-2} \frac{\partial u}{\partial \nu} - |\nabla v |^{p-2} \frac{\partial v}{\partial \nu}\right) \left(u + \varepsilon-v\right)^{-}d{\cal H}^{N-1}=$$ 
$$\int_{\Gamma_1} \left(|\nabla u |^{p-2} \frac{\partial u}{\partial \nu} - |\nabla v |^{p-2} \frac{\partial v}{\partial \nu}\right) \left(u+ \varepsilon -v\right)^{-}d {\cal H}^{N-1} + \int_{\Gamma_2} \left(|\nabla u |^{p-2} \frac{\partial u}{\partial \nu} - |\nabla v |^{p-2} \frac{\partial v}{\partial \nu}\right) \left(u+ \varepsilon-v\right)^{-}d {\cal H}^{N-1},$$

\noindent and we know that the integrand on $\Gamma_1$ is non positive since $\left(u+ \varepsilon-v\right)^{-} \le 0$, while the second integral is zero. Then the left hand side of \eqref{0.1Marilia} is non negative. this ensures that $\nabla \left(u + \varepsilon -v\right) \equiv 0$, so $v=u+ C$ in the set $\left\{u+ \varepsilon < v\right\}$. Hence $v \le u+ \varepsilon$ for any $\varepsilon$, and then $u\ge v$. 

\end{proof}

We can prove the following result.

\begin{Theorem}\label{asymptoticpto1equation}
Let $\Omega:=\omega \times \left(-\frac{1}{2},\frac{1}{2}\right)$ and assume that $\omega \subset \mathbb R^2$ is a bounded smooth domain  whose boundary has positive mean curvature, and let $u_0 \in C(\partial \omega)$. Then the unique weak solutions $u_p$ 
of \eqref{p-e-equation} in the sense that they are in $C(\overline{\Omega})\cap W^{1,p}_{\rm loc}(\Omega)$ and
\begin{equation}\nonumber
\displaystyle{\int_\Omega \left(\left|Id_\e \nabla u_p \cdot \nabla u_p\right|^{\frac{p-2}{2}}\ Id_\e \nabla u_p \right)\cdot \nabla \varphi dx=0}
\end{equation}
for every $\varphi \in C^\infty_0(\Omega)$,
are also $p-harmonic$ functions referred to \eqref{p-0-equation-1} and \eqref{p-0-equation-2} and, converge uniformly as $p \to 1$ to the unique function of least gradient in $\omega$ with datum $u_0$.   
\end{Theorem}

\begin{proof}[Proof]
For $p >1$, the existence and uniqueness of $p-harmonic$ solutions (independent on $x_3$) to \eqref{p-0-equation-1} and \eqref{p-0-equation-2} can be deduced as already observed in Theorem \ref{thm3.1J}. 
For $p=1$, we observe that Theorem \ref{thm2.1J} applied to $\omega$ ensures that there exists a unique function $u$ of least gradient with datum $u_0$ . Moreover again Theorem \ref{thm3.1J} and \cite[Remark 3.4]{J} provide the uniform convergence of the above $u_p$ to this solution $u$. To conclude the proof it remains to show that $u_p$ are also unique among the functions in $W^{1,p}(\Omega)\cap C(\overline{\Omega})$.  This latter fact follows from the lemma \ref{Marilia} applied to the unrescaled domain $\Omega_\e$.

\end{proof}

Now we can introduce the least gradient problem in the thin domain, taking into account the rescaling in \eqref{changeofvariable}
\begin{equation}\label{leastgradient-e-1}
\displaystyle{\inf_{u \in BV(\Omega)}\left\{\left|D_\alpha u, \frac{1}{\e }D_3 u\right|(\Omega), u \equiv u_0 \hbox{ on } \partial \omega \times \left(-\frac{1}{2}, \frac{1}{2}\right)\right\},}
\end{equation}
and its version on the class $BV(\Omega) \cap C(\overline{\Omega})$,
\begin{equation}\label{leastgradient-e-2}
\displaystyle{\inf_{u \in BV(\Omega)\cap C(\overline{\Omega})}
\left\{\left|D_\alpha u, \frac{1}{\e }D_3 u \right|(\Omega), u \equiv u_0 \hbox{ on } \partial \omega\times\left(-\frac{1}{2},\frac{1}{2}\right)\right\}.}
\end{equation}


  
In order to provide sufficient conditions ensuring that both problems \eqref{leastgradient-e-1} and \eqref{leastgradient-e-2} admit a unique solution, we prove the following theorem.

\begin{Theorem}\label{lastleast}
Let $\Omega:\omega \times \left(-\frac{1}{2},\frac{1}{2}\right)$ with $\omega \subset \mathbb R^2$ a bounded open set, piecewise $C^1$, and verifying \eqref{3.1SWZ}, \eqref{3.2SWZ} and a uniform exterior ball condition as in Theorem \ref{thm2.2SZ}. Let $u_0 \in C(\partial \omega)$. Then problems \eqref{leastgradient-e-1} and \eqref{leastgradient-e-2} admit a unique coincident solution, obtained as limit  for $p\to 1$ in $L^1(\Omega)$-strong topology and locally uniformly  in $\Omega$ of $\{u_{p,\e}\}$, where the latter is the unique solution of \eqref{p-e-equation}.
\end{Theorem}
\begin{proof}[Proof]
We start by observing that the assumptions on $\omega$ ensure that, as in \cite[Theorem 3.1]{J} $u_0 \in W^{1-\frac{1}{p},p}(\omega)$ for $1<p<2$. Consequently for every $1<p<2$ there exists a unique function $u_p$ solution of \eqref{p-0-equation-2} . The fact that $u_p$ is independent of $x_3$, implies that $u_p$ solves also \eqref{p-0-equation-1} and \eqref{p-e-equation} for every $\e>0$. On the other hand theorem \ref{asymptoticpto1equation} says also that $u_p$ is the unique solution of \eqref{p-e-equation}. Thus we can denote this solution $u_p$ also as $u_{p,\e}$. 
Next we can observe, by virtue of the strict convexity of $I_{p,\e}$ in \eqref{p-e-functional} and $I_{p,0}$ in \eqref{p-0-functional}, that for every $1<p<2$ and for every $\e>0$, $u_p \equiv u_{p,\e}$ is also the unique mimimum point of ${\cal P}_{p,0}$ and ${\cal P}_{p,\e}$. On the other hand Theorem \ref{asymptoticpto1equation} guarantees that $u_{p,\e}=u_p$ converges uniformly in $\overline{\Omega}$  (hence in $L^1(\Omega)$) to the unique solution $u$ of \eqref{leastgradient1} and \eqref{leastgradient2}. It is easily seen that the function $u$ is admissible also for problems \eqref{leastgradient-e-1} and \eqref{leastgradient-e-2}.
 Moreover the fact that $u_0$ is an admissible boundary datum for the $\Gamma$-convergence  theorems \ref{Gammap0to10} and \ref{corollary33AGZAimeta}, leads us to conclude that the common mimimum values $u_p$ of ${\cal P}_{p,\e}$ and ${\cal P}_{p,0}$ converge to the minimum of ${\cal P}_{1,0}$ and ${\cal P}_{1,\e}$. Consequently exploiting Theorem \ref{cor7.17DM} we can say that $u$  (the strong $L^1(\Omega)$ limit of $u_{p,\e}=u_p$ as $p \to 1$) is a minimum both for \eqref{leastgradient-e-1} and \eqref{leastgradient-e-2}. This  concludes the proof.       
\end{proof}

\section{Acknowledgements}

The authors acknowledge the support of Universit\`a degli Studi di Salerno and Universit\`a degli Studi del Sannio.
The last author is indebted with Irene Fonseca for the introduction to the topics of $-\Delta_1$ and the least gradient theory and, together with coauthors with Roberto Paroni and Carlo Sbordone for having proposed this investigation.

\end{document}